\documentclass[11pt,reqno]{amsart}

\usepackage{amsmath, amsthm, amsopn, amssymb, graphicx, enumerate, geometry, afterpage, caption, subcaption, color, textcomp, stmaryrd, bbm, hyperref}
\usepackage[english]{babel}
\usepackage[foot]{amsaddr}
\usepackage{csquotes}
\newtheorem{thm}{Theorem}[section]
\newtheorem{lemma}[thm]{Lemma}
\newtheorem{prop}[thm]{Proposition}

\newtheorem{cor}[thm]{Corollary}

\theoremstyle{definition}

\renewcommand{\subsubsection}[1] {\smallskip \noindent {\bf #1.}}
\newcommand{\reffig}[1] {\textsc{\ref{#1}}}

\newcommand{\cA}{\mathcal{A}}
\newcommand{\cU}{\mathcal{U}}

\newcommand{\cW}{\mathcal{W}}

\newcommand{\cP}{\mathcal{P}}
\newcommand{\cC}{\mathcal{C}}

\newcommand{\cO}{\mathcal{O}}
\newcommand{\cB}{\mathcal{B}}

\newcommand{\N}{\mathbb{N}}

\newcommand{\Z}{\mathbb{Z}}

\newcommand{\E}{\mathbb{E}}

\newcommand{\1}{\mathbbm{1}}

\newcommand{\intB}{\partial_{\bullet}}
\newcommand{\extB}{\partial_{\circ}}
\newcommand{\intextB}{\partial_{\ins \out}}

\newcommand{\ins}{\bullet}
\newcommand{\out}{\circ}

\newcommand{\up}{\,\uparrow\,}

\def\zero{\bold{0}}
\DeclareMathOperator\rev{rev}
\newcommand{\Even}{\mathrm{Even}}
\newcommand{\Odd}{\mathrm{Odd}}
\newcommand{\MC}{\mathrm{MC}}
\DeclareMathOperator\dist{dist}
\DeclareMathOperator\cut{\mathcal{S}}

\DeclareMathOperator{\Ocut}{\mathcal{S}}
\newcommand{\OC}[1]{{\Ocut}_{#1}}
\newcommand{\OCC}[2]{\mathrm{OddCut}_{#1}(#2)}

\makeatletter
\newcommand*\rel@kern[1]{\kern#1\dimexpr\macc@kerna}
\newcommand*\widebar[1]{%
  \begingroup
  \def\mathaccent##1##2{%
    \rel@kern{0.8}%
    \overline{\rel@kern{-0.8}\macc@nucleus\rel@kern{0.2}}%
    \rel@kern{-0.2}%
  }%
  \macc@depth\@ne
  \let\math@bgroup\@empty \let\math@egroup\macc@set@skewchar
  \mathsurround\z@ \frozen@everymath{\mathgroup\macc@group\relax}%
  \macc@set@skewchar\relax
  \let\mathaccentV\macc@nested@a
  \macc@nested@a\relax111{#1}%
  \endgroup
}
\makeatother

\title{The growth constant of odd cutsets in high dimensions}
\date{\today}

\author{Ohad Noy Feldheim\textsuperscript{\dag}}
\address{\dag\, Stanford University.
    Department of Mathematics.
    Stanford, CA 94305, U.S.A.}

\author{Yinon Spinka\textsuperscript{\ddag}}
\address{\ddag\, Tel Aviv University.
    School of Mathematical Sciences.
    Tel Aviv, 69978, Israel.}
    \email{ohadf@netvision.net.il, yinonspi@post.tau.ac.il}
\thanks{Research of Y.S. supported by Israeli Science Foundation grant 861/15, the European Research Council starting grant 678520 (LocalOrder), and the Adams Fellowship Program of the Israel Academy of Sciences and Humanities}

\begin{document}

\begin{abstract}
A cutset is a non-empty finite subset of $\Z^d$ which is both connected and co-connected. A cutset is odd if its vertex boundary lies in the odd bipartition class of $\Z^d$.
Peled~\cite{peled2010high} suggested that the number of odd cutsets which contain the origin and have $n$ boundary edges may be of order $e^{\Theta(n/d)}$ as $d \to \infty$, much smaller than the number of general cutsets, which was shown by Lebowitz and Mazel~\cite{lebowitz1998improved} to be of order $d^{\Theta(n/d)}$.
In this paper, we verify this by showing that the number of such odd cutsets is $(2+o(1))^{n/2d}$.
\end{abstract}

\maketitle

\section{Introduction and results}

We consider the integer lattice $\Z^d$ as a graph with nearest-neighbor adjacency, i.e., the edge set is the set of $\{u,v\}$ such that $u$ and $v$ differ by one in exactly one coordinate.
The \emph{edge-boundary} of a subset $U$ of $\Z^d$ is the set of edges having exactly one endpoint in $U$, and the \emph{internal vertex-boundary} of $U$ is the set of vertices in $U$ which are adjacent to some vertex outside $U$.

A \emph{cutset} is a non-empty finite subset of $\Z^d$ which is both connected and co-connected (i.e., both it and its complement span connected subgraphs).
The edge-boundaries of cutsets are exactly the finite minimal edge-cuts of $\Z^d$, i.e., finite minimal sets of edges whose removal disconnects $\Z^d$.
A vertex of $\Z^d$ is called {\em odd (even)} if it is at odd (even) graph-distance from the origin, and a subset of $\Z^d$ is called odd (even) if its internal vertex-boundary consists solely of odd (even) vertices.
In this work, we study $\OCC{n}{d}$, the number of odd cutsets in $\Z^d$ with edge-boundary size $n$ which contain the origin. Random samples of such sets are depicted in Figure~\ref{fig:oddcut-sample}. Our main result is the following.

\begin{thm}\label{thm:main}
	There exists a constant $C>0$ such that for any integer $d \ge 2$ and any sufficiently large multiple $n$ of $2d$, we have
	\[ 2^{\frac{n}{2d} \big(1 + 2^{-2d} \big)} \le \OCC{n}{d} \le 2^{\frac{n}{2d} \big(1 + \frac{C \log^{3/2} d}{\sqrt{d}} \big)} .\]
\end{thm}

We further prove the existence of a growth constant for the number of odd cutsets.

\begin{thm}\label{thm:limit-exists}
	For any integer $d \ge 2$, the limit $\mu(d) := \displaystyle\lim_{n \to \infty} \OCC{2dn}{d}^{1/n}$ exists.
\end{thm}

The existence of the above limit is proven via a super-multiplicitivity argument. It follows from Theorem~\ref{thm:main} that $\mu(d)$ satisfies the following bounds:
\[ 1+2^{-2d} \le \log_2 \mu(d) \le 1+\frac{C \log^{3/2}d}{\sqrt{d}} .\]
We remark that the divisibility condition on $n$ in the theorems is essential as the size of the edge-boundary of an odd set in $\Z^d$ is always a multiple of $2d$ (see Lemma~\ref{lem:odd-set-boundary-size} below).

The lower bound in Theorem~\ref{thm:main} is obtained with relative ease, by estimating the number of odd cutsets which are obtained as local fluctuations of a single set, bearing resemblance to a $(d-1)$-dimensional cube. Note that some restriction on the minimum value of $n$ is necessary, since any odd cutset $S$ in $\Z^d$ which contains the origin has at least $2d(2d-1)$ boundary edges (see Corollary~\ref{cor:odd-set-min-boundary-size} below).

The upper bound in Theorem~\ref{thm:main}, which is the main result of this paper, is obtained by a more involved method.
It is based on the intuition that the primary phenomenon which accounts for the number of odd cutsets is the great variety of possible local structures near the boundary. In other words, every odd cutset can be obtained as a perturbation of one of a relatively small number of global shapes. Thus, the proof of the upper bound is based on a classification of odd cutsets according to their approximate global structure, which we call an approximation. We first show that the number of different approximations is small and then provide tight bounds on the number of regular odd sets corresponding to each approximation and use it to bound the total number of odd cutsets.
This general method of approximations goes back to Sapozhenko~\cite{sapozhenko1987onthen}. Similar methods were used also by Peled~\cite{peled2010high}, by Galvin and Kahn~\cite{galvin2004phase} and by the authors~\cite{feldheimspinka}. The proof given here relies on ideas from~\cite{sapozhenko1987onthen} and follows the approach of~\cite{galvin2004phase} with simplifications introduced in~\cite{feldheimspinka} (in a more complex setting).
As it requires no additional effort, we prove the upper bound under weaker connectivity assumptions than those used in the definition of a cutset (see Theorem~\ref{thm:strong-upper}).

We remark that although Theorem~\ref{thm:main} is stated (and has meaningful content) for all $d \ge 2$, the bounds become crude when $d$ is small, in which case a similar upper bound could be obtained from the bound in~\cite{lebowitz1998improved} for general cutsets.

\begin{figure}
	\centering
	\begin{subfigure}[t]{.5\textwidth}
		\centering
		\includegraphics[scale=0.036]{oddcut-sample1.pdf}

		\caption{$n=3000$}
	\end{subfigure}%
	\begin{subfigure}[t]{.5\textwidth}
		\centering
		\includegraphics[scale=0.12]{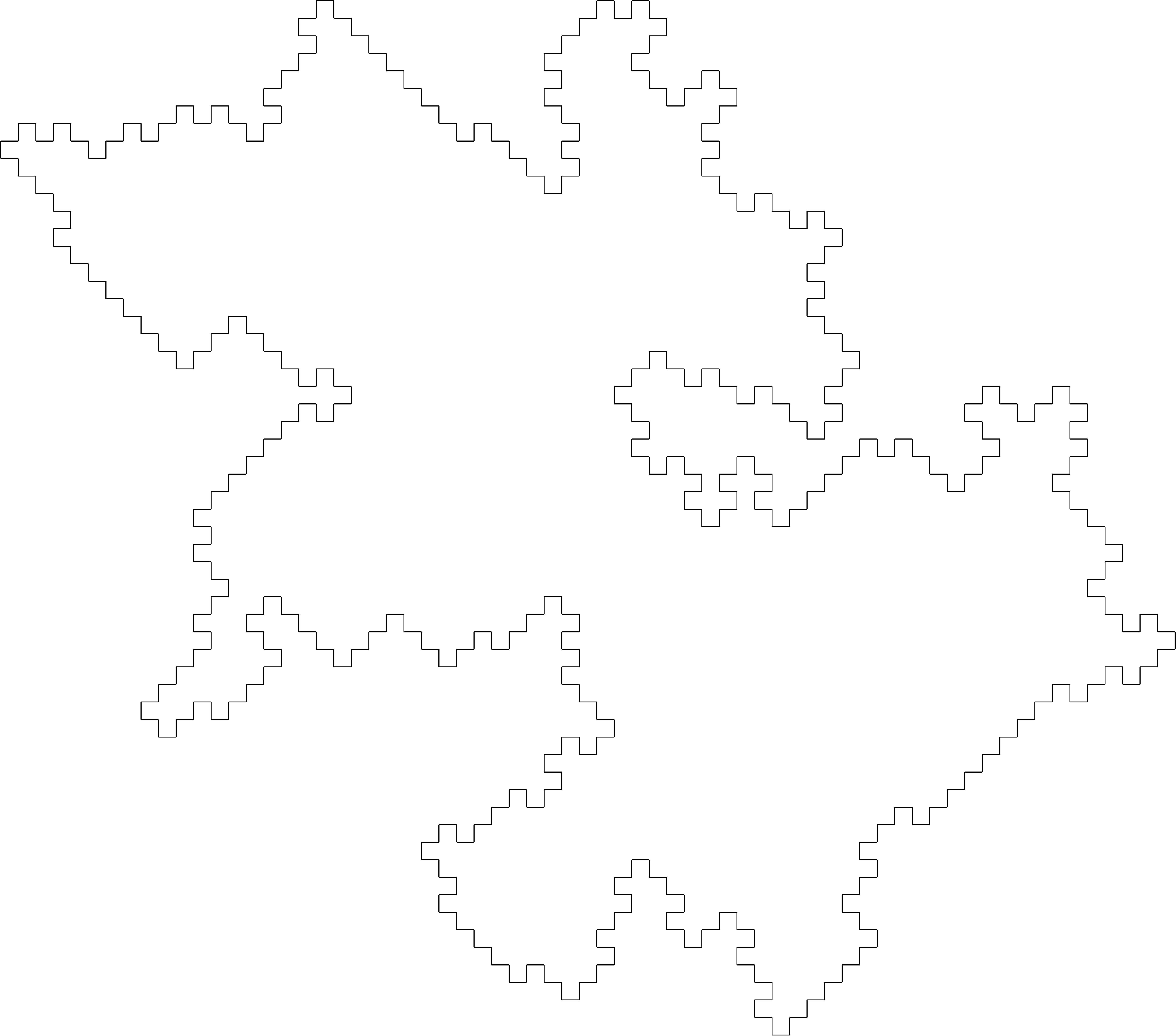}

		\caption{$n=600$}
	\end{subfigure}
	\caption{Samples of random odd cutsets in $\Z^2$ with $n$ boundary edges.}
	\label{fig:oddcut-sample}
\end{figure}

\subsection{Discussion}

In 1988, Lebowitz and Mazel~\cite{lebowitz1998improved} investigated general
cutsets in $\Z^d$ (which they refer to as \emph{primitive Peierls contours}).
They showed that the number of cutsets with boundary size $n$ which contain the origin is at most $d^{64n/d}$ when $d \ge 2$, and used this to show that the low-temperature expansion for the $d$-dimensional Ising model, written in terms of Peierls contours, converges when the inverse-temperature is at least $64(\log d)/d$.
Ten years later, Balister and Bollob\'as~\cite{balister2007counting} improved this result by reducing the aforementioned bound on the number of cutsets to $(8d)^{2n/d}$. They also proved that the number of such sets is bounded from below by $(cd)^{n/2d}$.

Odd cutsets have been used in various probabilistic models to obtain phase transition and torpid mixing results. Some of these include works on the hard-core model by Galvin~\cite{galvin2008sampling}, Galvin--Kahn~\cite{galvin2004phase}, Galvin--Tetali~\cite{galvin2004slow,galvin2006slow} and Peled--Samotij~\cite{peled2014odd}, on homomorphism height functions by Galvin~\cite{Galvin2003hammingcube} and on 3-colorings by Galvin~\cite{galvin2007sampling}, Galvin--Randall~\cite{galvin2007torpid}, Galvin--Kahn--Randall--Sorkin~\cite{galvin2012phase} and Peled~\cite{peled2010high} (who also treated discrete Lipschitz functions).
Recently, using a generalization of odd cutsets, the authors showed that the 3-state antiferromagnetic Potts model in high dimensions undergoes a phase transition at positive temperature.

Peled~\cite{peled2010high} raised the question of whether or not the number of odd cutsets is of smaller order of magnitude than the total number of cutsets. Namely, he asked whether this quantity is of order $d^{\Theta(n/d)}$ or only of order $e^{\Theta(n/d)}$. Theorem~\ref{thm:main} resolves this question by showing that it is indeed the latter and pinpointing the constant in the exponent, i.e., $(2+o(1))^{n/2d}$.

It is worthwhile to mention that the method of approximations, which we use to obtain our upper bound, played a role in many of the aforementioned works. This method goes back to Sapozhenko who studied enumeration problems on bipartite graphs and posets~\cite{sapozhenko1987onthen,sapozhenko1989number,sapozhenko1991number} motivated by previous results of Korshunov on antichains~\cite{korshunov1981number} and of Korshunov--Sapozhenko on binary codes~\cite{Korshunov1983Th}.

In addition to cutsets, other types of connected subgraphs of $\Z^d$ have also been investigated. In this context, we mention the recent work of Miranda--Slade~\cite{miranda2010growth} who obtained estimates for the growth constant of lattice trees and lattice animals in high dimensions.

\subsection{Open problems}

The bounds obtained in Theorem~\ref{thm:main} on the number of odd cutsets match in the first order term at the exponent. The next order term is determined by $\lambda(d) := \log_2 \mu(d) - 1$.
We have shown that $2^{-2d} \le \lambda(d) \le \tfrac{C \log^{3/2}d}{\sqrt{d}}$ and it is natural to ask what the correct asymptotics of $\lambda(d)$ is. Namely, is it exponential as in the lower bound? Is it polynomial as in the upper bound?

In~\cite{peled2010high}, Peled also raised the question of determining the scaling limit of odd cutsets.
He suggested that in contrast to the case of ordinary cutsets (without the oddness condition), where it is plausible that the scaling limit is super Brownian motion, it may be the case that a random odd cutset typically contains a macroscopic cube in its interior.

\subsection{Notation}\label{sec:notation}

\begin{figure}
	\centering
	\begin{subfigure}[t]{.27\textwidth}
		\centering
		\includegraphics[scale=0.45]{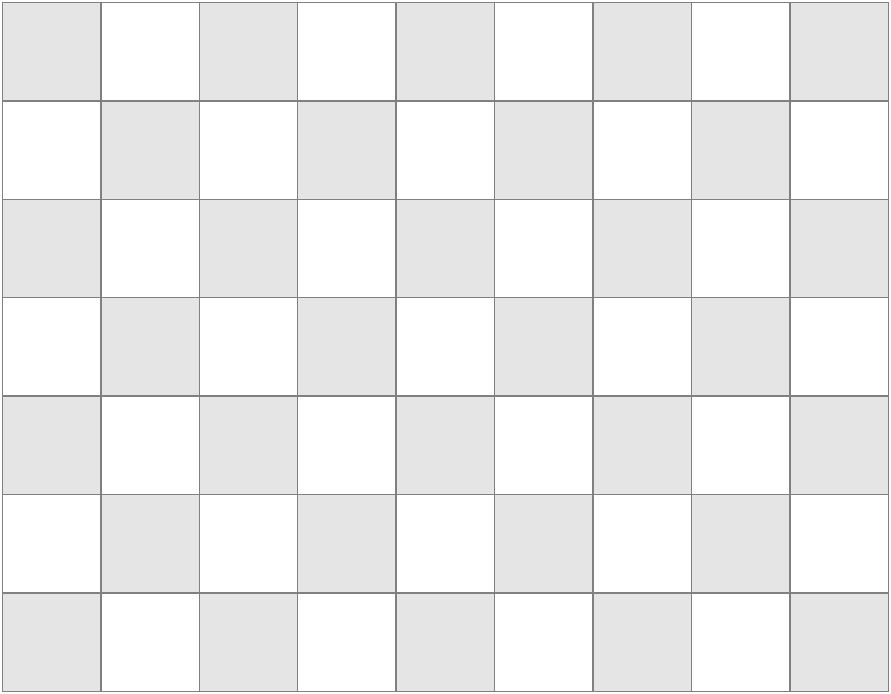}
		\caption{The even and odd bipartition classes of $\Z^d$.}
		\label{fig:even-odd-bipartition}
	\end{subfigure}%
	\begin{subfigure}{20pt}
		\quad
	\end{subfigure}%
	\begin{subfigure}[t]{.27\textwidth}
		\centering
		\includegraphics[scale=0.45]{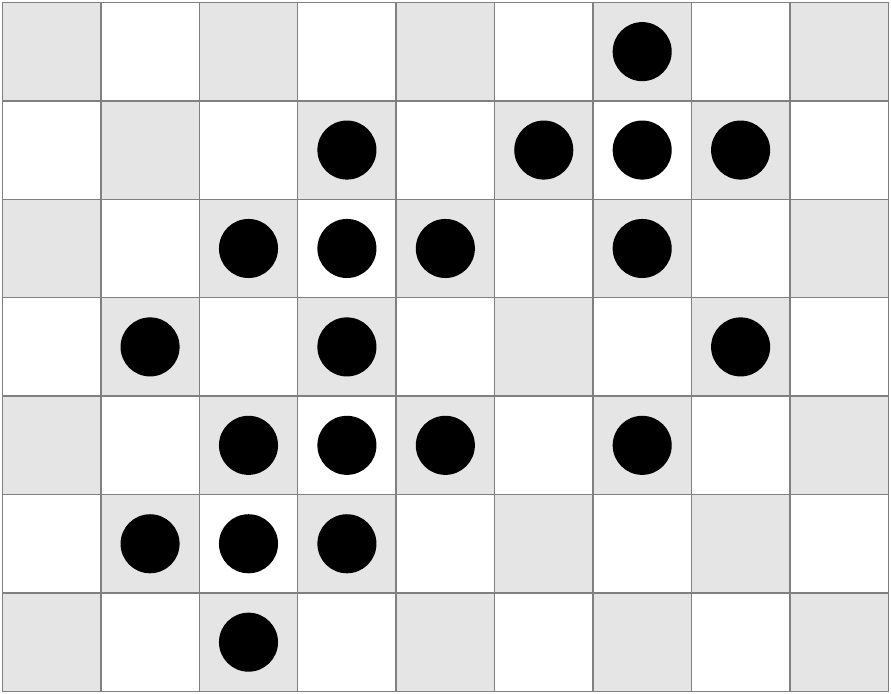}
		\caption{An odd set.}
		\label{fig:odd-set}
	\end{subfigure}%
	\begin{subfigure}{20pt}
		\quad
	\end{subfigure}%
	\begin{subfigure}[t]{.27\textwidth}
		\centering
		\includegraphics[scale=0.45]{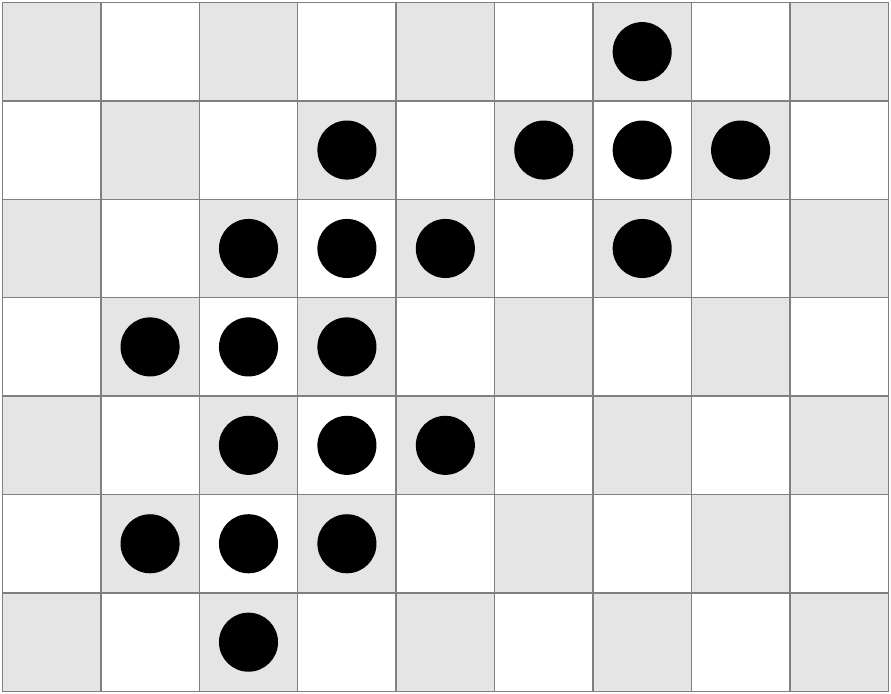}
		\caption{A regular odd set.}
		\label{fig:regular-odd-set}
	\end{subfigure}
	\caption{The vertices of $\Z^d$ are depicted as squares, with the even vertices in white and the odd vertices in gray. An odd set is a set whose internal boundary consists solely of odd vertices. A set is regular if both it and its complement have no isolated vertices.}
	\label{fig:odd-sets}
\end{figure}

Let $G=(V,E)$ be a graph.
For vertices $u,v \in V$ such that $\{u,v\} \in E$, we say that $u$ and $v$ are {\em adjacent} and write $u \sim v$.
For a subset $U \subset V$, denote by $N(U)$ the {\em neighbors} of $U$, i.e., vertices in $V$ adjacent to some vertex in $U$, and define for $t>0$,
\[ N_t(U) := \{ v \in V : |N(v) \cap U| \ge t \} .\]
In particular, $N_1(U)=N(U)$.
Denote the \emph{internal} and \emph{external vertex-boundary} of $U$ by $\intB U := U \cap N(U^c)$ and $\extB U := N(U) \setminus U$, respectively. We also use the notation $\intextB U := \intB U \cup \extB U$, $U^+ := U \cup \extB U$ and $v^+ := \{v\}^+$.
The set of edges between two disjoint sets $U$ and $W$ is denoted by $\partial(U,W):=\{ \{u,w\} \in E : u \in U, w \in W \}$.
In particular, the \emph{edge-boundary} of $U$ is $\partial U := \partial(U,U^c)$. We also write $\partial u := \partial \{u\}$.
The graph-distance between $u$ and $v$ is denoted by $\text{dist}(u,v)$. For two non-empty sets $U,W \subset V$, we denote by $\dist(U,W)$ the minimum graph-distance between a vertex in $U$ and a vertex in $W$.

{\bf Policy regarding constants.} In the rest of the paper, we employ the following policy on constants. We write $C,c,C',c'$ for positive absolute constants, whose values may change from line to line. Specifically, the values of $C,C'$ may increase and the values of $c,c'$ may decrease from line to line.

\subsection{Odd sets and isoperimerty}
We use $\Even$ ($\Odd$) to denote the set of even (odd) vertices of $\Z^d$.
Thus, a set $U \subset \Z^d$ is odd if and only if $\intB U \subset \Odd$ and it is even if and only if $\intB U \subset \Even$.
We say that $U$ is \emph{regular} if both it and its complement contain no isolated vertices.
Thus, a cutset is regular if and only if it is not a singleton.
Observe that $U$ is odd if and only if $(\Even \cap U)^+ \subset U$ and that $U$ is regular odd if and only if $U=(\Even \cap U)^+$ and $U^c=(\Odd \cap U^c)^+$.
See Figure~\reffig{fig:odd-sets}.

For a set $U\subset\Z^d$ and a unit vector $s \in \Z^d$, we define the boundary of $U$ \emph{in direction $s$} to be $\partial^s U := \{v \in U : v+s \notin U \}$.
A nice property of odd sets is that the size of their boundary is the same in every direction.

\begin{lemma}\label{lem:odd-set-boundary-size}
	Let $U \subset \Z^d$ be finite and odd.
	Then, for any unit vector $s \in \Z^d$, we have
	\[ |\partial^s U| = |\Odd \cap U| - |\Even \cap U| = \tfrac{|\partial U|}{2d} .\]
\end{lemma}
\begin{proof}
Denote $U^s := \{ u+s : u \in U \}$.
The first equality follows from
\begin{align*}
|\Even \cap U| &= |\Odd \cap U^s| = |\Odd \cap U^s \cap U| = |\Odd \cap U| - |\Odd \cap U \setminus U^s| \\&= |\Odd \cap U| - |U \setminus U^s| = |\Odd \cap U| - |\partial^s U| .
\end{align*}
The second equality now follows from the first, since $|\partial U| = \sum_{s'} |\partial^{s'} U| = 2d \cdot |\partial^s U|$.
\end{proof}

\begin{cor}\label{cor:odd-set-min-boundary-size}
Let $U \subset \Z^d$ be finite and odd. If $U$ contains an even vertex then $|\partial U| \ge 2d(2d-1)$.
\end{cor}
\begin{proof}
Let $u \in U$ be even. Since $U$ is odd, we have $u^+ \subset U$. Thus, $|\partial^s U| \ge 2d-1$, where $s \in \Z^d$ is any unit vector, and the corollary follows from Lemma~\ref{lem:odd-set-boundary-size}.
\end{proof}

We conclude with a well-known isoperimetric inequality.

\begin{lemma}[\cite{feldheimspinka}]\label{lem:isoperimetry}
	Let $U \subset \Z^d$ be finite. Then $|\partial U| \ge 2d \cdot |U|^{1-1/d}$.
\end{lemma}

\subsection{Organization}
The rest of the paper is organized as follows. In Section~\ref{sec:super-mult}, we show that $\OCC{n}{d}$ is almost super-multiplicitive and use this to prove Theorem~\ref{thm:limit-exists}. In Section~\ref{sec:lower-bound}, we prove the lower bound stated in Theorem~\ref{thm:main}. In Section~\ref{sec:upper-bound}, we state two propositions; Proposition~\ref{lem:number-of-cutsets-with-approx} which bounds the number of odd cutsets approximated by a given approximation and Proposition~\ref{lem:family-of-approx} which shows that a relatively small number of approximations are sufficient to approximate every odd cutset in question. We then deduce the upper bound stated in Theorem~\ref{thm:main} from these propositions. Section~\ref{sec:number-of-cutsets-with-approx} and Section~\ref{sec:approx} are dedicated to the proofs of Proposition~\ref{lem:number-of-cutsets-with-approx} and Proposition~\ref{lem:family-of-approx}, respectively.

\subsection{Acknowledgments}
We wish to thank Ron Peled for suggesting the problem to us and for useful discussions.

%
\section{Almost super-multiplicitivity}
\label{sec:super-mult}
%

The main step in showing the existence of the limit defining the growth constant $\mu(d)$ is establishing the following ``almost'' super-multiplicitivity property of $\OCC{n}{d}$.

\begin{prop}\label{prop:almost-super-mult}
Let $d \ge 2$ and let $n,m,k \in 2d \N$ with $k \ge 12d^2$. Then
\[ \OCC{n+m+k}{d} \ge \frac{\OCC{n}{d} \cdot \OCC{m}{d}}{ \left(\frac{m}{d}\right)^{\frac{d}{d-1}}} .\]
\end{prop}
\begin{proof}
Fix $d \ge 2$ and denote by $\cC_n$ the collection of odd cutsets $S$ in $\Z^d$ having $|\partial S|=n$ and containing the origin so that $|\cC_n|=\OCC{n}{d}$.
Endow $\Z^d$ with the partial order induced by the sum of the first two coordinates.
We say that a vertex $u$ is the \emph{peak} of an odd set $S$ if it is the unique maximal element among all the \emph{even} vertices in $S$.
Denote by $\cP_n$ the collection of odd cutsets in $\cC_n$ having a peak and by $\cO_n$ those having a peak at the origin. The proof consists of four parts:
\begin{enumerate}
	\item $|\cC_n| \le |\cP_{n+2d(2d-3)}|$.
	\item $|\cP_n| \le |\cO_n| \cdot (n/2d)^{d/(d-1)}$.
	\item $|\cO_n| \le |\cO_{n+k-8d^2}|$.
	\item $|\cP_n| \cdot |\cO_m| \le |\cC_{n+m-4d}|$.
\end{enumerate}
Since we may assume that $m \ge 2d(2d-1)$ by Corollary~\ref{cor:odd-set-min-boundary-size}, the proposition then follows from
\begin{align*}
|\cC_n| \cdot |\cC_m|
 &\overset{(1)}{\le} |\cP_{n+2d(2d-3)}| \cdot |\cP_{m+2d(2d-3)}| \\
 &\overset{(2)}{\le} |\cP_{n+2d(2d-3)}| \cdot |\cO_{m+2d(2d-3)}| \cdot (m/2d +2d-3)^{d/(d-1)} \\
 &\overset{(3)}{\le} |\cP_{n+2d(2d-3)}| \cdot |\cO_{m+k-4d^2+10d}| \cdot (m/d)^{d/(d-1)} \\
 &\overset{(4)}{\le} |\cC_{n+m+k}| \cdot (m/d)^{d/(d-1)} .
\end{align*}

\begin{figure}
	\centering
	\begin{subfigure}[t]{.27\textwidth}
		\centering
		\includegraphics[scale=0.45]{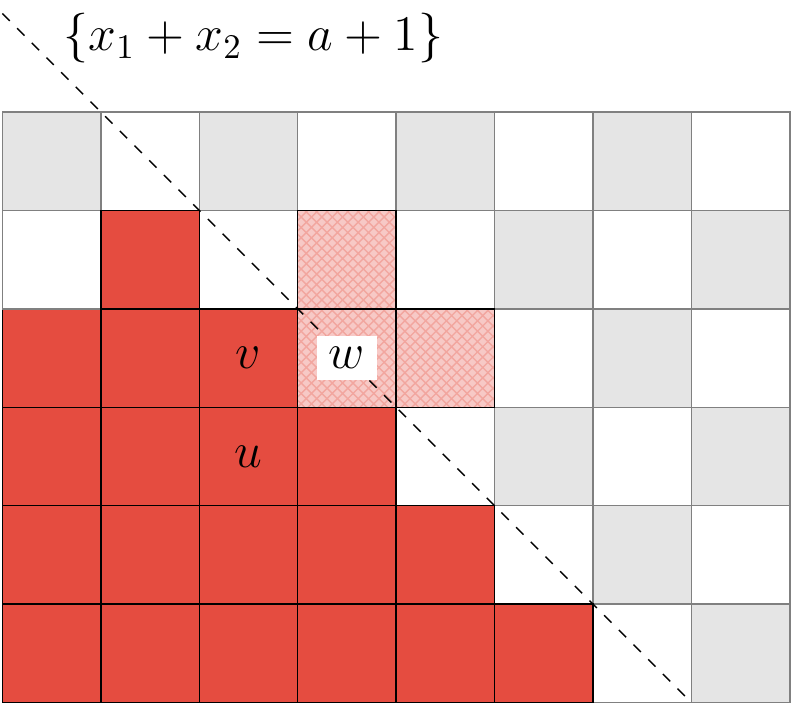}
		\caption{Modifying an odd cutset to create a peak.}
		\label{fig:creating-peak}
	\end{subfigure}%
	\begin{subfigure}{20pt}
		\quad
	\end{subfigure}%
	\begin{subfigure}[t]{.27\textwidth}
		\centering
		\includegraphics[scale=0.45]{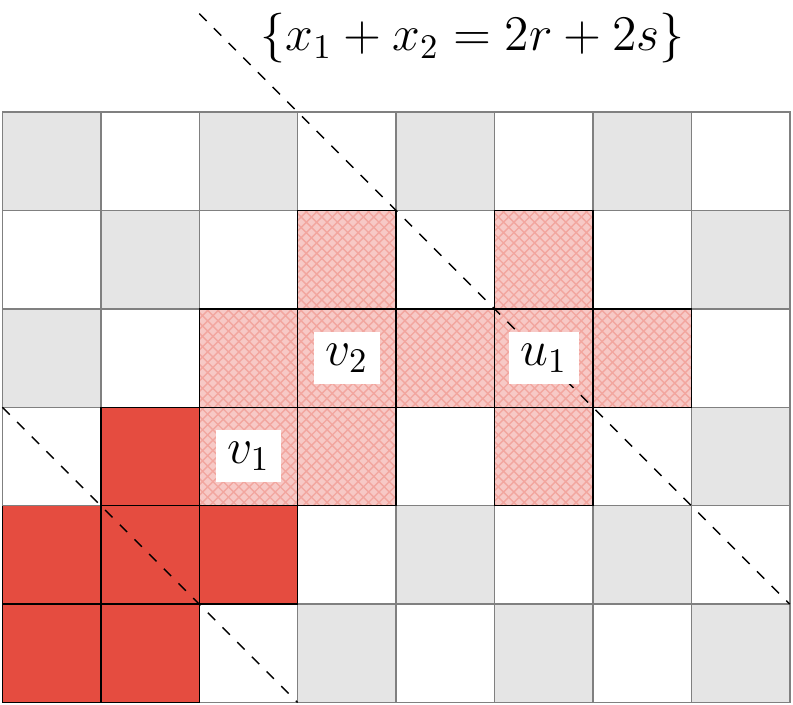}
		\caption{Modifying an odd cutset (with a peak) to adjust its boundary size.}
		\label{fig:extending-peak}
	\end{subfigure}%
	\begin{subfigure}{20pt}
		\quad
	\end{subfigure}%
	\begin{subfigure}[t]{.27\textwidth}
		\centering
		\includegraphics[scale=0.45]{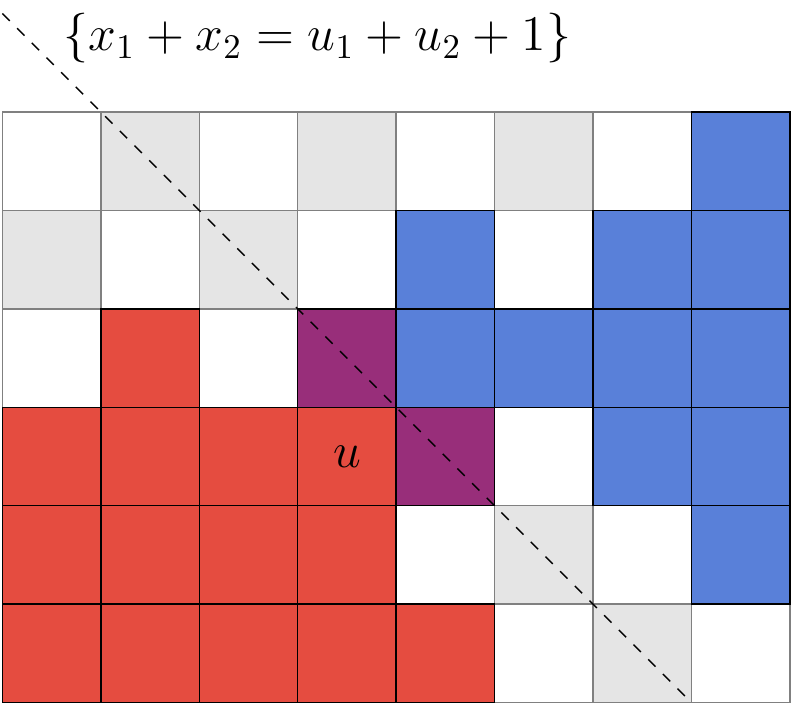}
		\caption{Merging two odd cutsets (one with a peak and one with an inverted peak).}
		\label{fig:merging-peak}
	\end{subfigure}
	\caption{Modifying odd cutsets.}
	\label{fig:modifying-cutsets}
\end{figure}

To prove the first part, we take a set $S \in \cC_n$ and construct a set $S' \in \cP_{n+2d(2d-3)}$ in an injective manner.
Let $a$ be the maximum value for which the hyperplane $\{ x_1+x_2 = a \}$ intersects $S$. Let $v$ be a vertex in this intersection, let $u \in S$ be a adjacent to $v$ (note that $S \neq \{v\}$ since $\zero^+ \subset S$) and denote $w := u+e_1+e_2 \notin S$. Since $S$ is odd and $v+e_1 \notin S$, we have that $v$ is odd, $u$ is even, $u^+ \subset S$ and $v-u \in \{e_1,e_2\}$. It is easy to check that $S' := S \cup w^+$ is an odd set with $|\partial S'|=n+2d(2d-3)$ and a peak at $w$. Since $S'$ is clearly connected, it remains only to check that $S'$ is co-connected. Since $S$ is co-connected, any two vertices $x,y \in (S')^c$ can be connected via a path in $S^c$. Thus, it suffices to check that any two vertices $x,y \in (S')^c \cap N(S' \setminus S)$ can be connected via a path in $(S')^c$. Using that $(S')^c \cap N(S' \setminus S) = w^{++} \cap \{x_1+x_2 \ge a+1\}$ and that $w$ is the peak of $S'$, this is easily verified.
See Figure~\reffig{fig:creating-peak}.

The second part follows from the isoperimetric inequality in Lemma~\ref{lem:isoperimetry}.

To prove the third part, let $r,s \in \N$, and, for $S \in \cO_n$, define
\[ S' := S \cup \big\{(i,i,0,\dots,0)\in\Z^d : 0< i\le r \big\}^+ \cup \big\{(r+2i,r,0,\dots,0)\in\Z^d : 0< i\le s \big\}^+ .\]
It is straightforward to check that $S' \in \cC_{n+r(2d-3)+s(2d-2)}$ and has a peak at $(r+2s,r,0,\dots,0)$. Since the mapping $S \mapsto S'$ is injective, this shows that $|\cO_n| \le |\cO_{n+r(2d-3)+s(2d-2)}|$. See Figure~\reffig{fig:extending-peak}.
Finally, since $2d-2$ and $2d-3$ are co-prime (i.e., their $\gcd$ is 1) and since $k-8d^2>(2d-3)(2d-2)$, it is possible to choose $r,s\in \N$ so that $k-8d^2 = r(2d-3)+s(2d-2)$
(this is known as Sylvester's solution to the Frobenius problem for two coins).

To prove the fourth part, we take an element $(S,S') \in \cP_n \times \cO_m$ and construct a set $T \in \cC_{n+m-4d}$ in an injective manner. Let $S''$ be the reflection of $S'$ through the hyperplane $\{x_1+x_2=0\}$ (i.e., $S''$ is obtained by negating the first two coordinates of every vertex in $S'$). Let $u$ be the peak of $S$ and define $T := S \cup R$, where $R:=u+e_1+e_2+S''$. Since $S$ and $R$ lie on different sides of the hyperplane $\{x_1+x_2=u_1+u_2+1\}$ (except for $\{u+e_1,u+e_2\}$, which is their common intersection with this hyperplane), it follows easily that $T$ is a connected odd set with $|\partial T|=n+m-4d$. To see that $T \in \cC_{n+m-4d}$, it remains to check that $T$ is co-connected, i.e., that $T^c = S^c \cap R^c$ is connected. This follows from the observation that $\widebar{S} := S^c \cap \{ x_1+x_2 \le u_1+u_2+1\}$ is connected, and similarly, that $\widebar{R} := R^c \cap \{ x_1+x_2 \ge u_1+u_2+1 \}$ is connected, and from $\widebar{S} \cap \widebar{R} \neq \emptyset$ and $T^c = \widebar{S} \cup \widebar{R}$. Finally, it is clear that this mapping is injective, since $u$ can be recovered by considering all hyperplanes $\{x_1+x_2=a\}$ which intersect $T$ at two points and using that $|\partial S|$ is known. See Figure~\reffig{fig:merging-peak}.
\end{proof}

\begin{proof}[Proof of Theorem~\ref{thm:limit-exists}]
Fix $d \ge 2$ and denote $a_n := \OCC{2dn}{d}$ and $b_n := a_{n-6d} / 16n^2$. Clearly, it suffices to show the existence of the limit $\lim_{n \to \infty} b_n^{1/n}$. This will follow from Fekete's subadditive lemma (applied to $-\log b_n$) if we show that $b_n$ is a super-multipliciative sequence. Indeed, by Proposition~\ref{prop:almost-super-mult}, for $n \ge m > 6d$, we have
\[ b_{n+m} = \frac{a_{(n-6d)+(m-6d)+6d}}{16(n+m)^2} \ge \frac{a_{n-6d} \cdot a_{m-6d}}{16(n+m)^2 \cdot (2m)^{d/(d-1)}} \ge b_n b_m \cdot \frac{4n^2}{(n+m)^2} \ge b_n b_m . \qedhere \]
\end{proof}

\section{The lower bound}
\label{sec:lower-bound}

The proof of the lower bound in Theorem~\ref{thm:main} is based on a simple counting argument. The idea appeared already in~\cite{peled2010high} (see also~\cite{peled2014odd}). Since the details have not appeared in print, we give a short proof here. Let $d \ge 2$ and let $m$ be a large even integer.
We first prove the lower bound for $n := 2d(m^{d-1}+(d-1)m^{d-2})$ directly by constructing a large family of odd cutsets having $n$ boundary edges. We then use Proposition~\ref{prop:almost-super-mult} to extend the lower bound to other values of $n$.
For brevity, we shall employ the notation $[a,b):=\{a,\dots, b-1\}$ for integers $a<b$.
The proof is accompanied by Figure~\ref{fig:lower-bound}.

Let $B_0 := \Even \cap [0,m)^{d-1} \times \{0\}$ and observe that $B_0^+$ is an odd cutset in $\Z^d$ which contains the origin.
We now show that its edge-boundary size is $n$.
Let $\up=e_d$ be the $d$-th standard basis vector and recall that the boundary of $U \subset \Z^d$ in direction $\up$ is $\partial^\uparrow U = \{v \in U : v+e_d \notin U \}$, and that, by Lemma~\ref{lem:odd-set-boundary-size}, $|\partial U|=2d|\partial^\uparrow U|$ if $U$ is odd.
Let $\pi \colon \Z^d \to \Z^{d-1}$ be the projection onto the first $d-1$ coordinates and note that $\pi(U)=\pi(\partial^\uparrow U)$ for any finite set $U$.
Observe also that
\[ \pi(B_0^+) = [0,m)^{d-1} \cup \bigcup_{i=0}^{d-2}
\Odd_{d-1} \cap [0,m)^{i}\times \{-1,m\} \times [0,m)^{d-2-i}, \]
where $\Odd_{d-1}$ denotes the set of odd vertices in $\Z^{d-1}$.
Thus, since $\pi$ is injective on $\partial^\uparrow B_0^+$, we have
\[ \frac{|\partial B_0^+|}{2d} = |\partial^\uparrow B_0^+| = |\pi(\partial^\uparrow B_0^+)| = |\pi(B_0^+)| = m^{d-1} + (d-1)m^{d-2} = \frac{n}{2d} .\]

\begin{figure}
	\centering
	\includegraphics[scale=0.45]{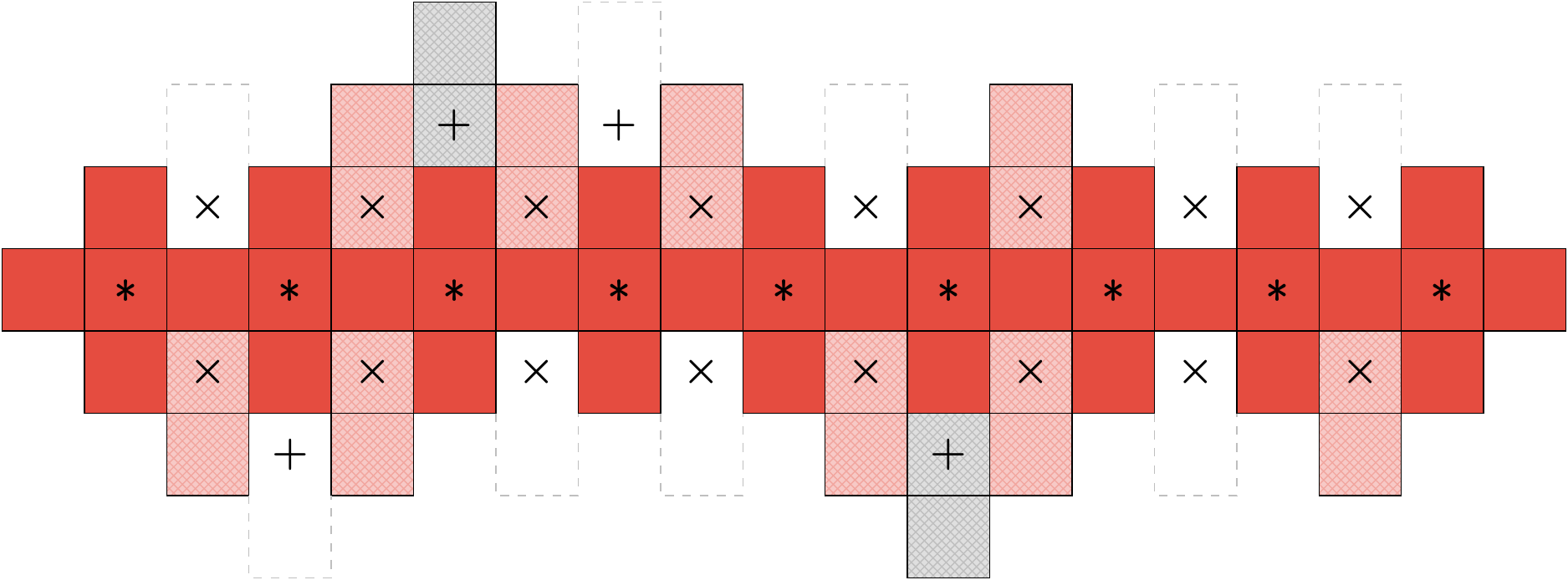}
	\caption{Constructing a large family of odd cutsets. The vertices in $B_0$, $A_1$ and $A_2(B_1)$ (all of which are even vertices) are depicted by $*$, $\times$ and +, respectively. The set $B_0^+$ is an odd cutset (shown in red), and any choice of $B_1 \subset A_1$ gives an odd cutset $(B_0 \cup B_1)^+$ with the same number of boundary edges (the red and pink regions). Once $B_1$ is chosen, any choice of $B_2 \subset A_2(B_1)$ then gives another odd cutset $(B_0 \cup B_1 \cup B_2)^+$ also with the same number of boundary edges (the red, pink and gray regions).}
	\label{fig:lower-bound}
\end{figure}

Let $A_1 := \Even \cap [1,m-1)^{d-1} \times \{\pm 1\}$ and observe that for any $B_1 \subset A_1$, the set $(B_0 \cup B_1)^+$ is an odd cutset. Since $\pi((B_0 \cup B_1)^+)=\pi(B_0^+)$ and since $\pi$ is injective on $\partial^\uparrow(B_0 \cup B_1)^+$, we deduce that $|\partial(B_0 \cup B_1)^+|=|\partial B_0^+|=n$.
Also notice that $(v_1,\dots, v_{d-1},2)\in (B_0\cup B_1)^+$ if and only if $(v_1,\dots, v_{d-1},1)\in B_1$, and similarly, $(v_1,\dots, v_{d-1},-2)\in (B_0\cup B_1)^+$ if and only if $(v_1,\dots, v_{d-1},-1)\in B_1$, so that different choices of $B_1$ produce distinct sets. By counting the number of such sets, we obtain
\[ \OCC{n}{d} \ge 2^{|A_1|} = 2^{(m-2)^{d-1}} = 2^{\frac{n}{2d} \cdot \frac{(m-2)^{d-1}}{m^{d-1} +(d-1) m^{d-2}}} \ge 2^{\frac{n}{2d} \cdot (1-10d/m)} .\]
To obtain a better bound, we consider a ``second order'' augmentation.
Given $B_1 \subset A_1$, define
\[ A_2(B_1) := \left\{ x \in \Even \cap [2,m-2)^{d-1} \times \{\pm 2 \} ~:~ \forall 1 \le i \le d-1 ~~ (x_1,\dots,x_{d-1},\tfrac{x_d}{2}) \pm e_i \in B_1 \right\} .\]
As before, one may easily check that for every choice of $B_1 \subset A_1$ and $B_2 \subset A_2(B_1)$, the set $(B_0 \cup B_1 \cup B_2)^+$ is a distinct odd cutset with precisely $n$ boundary edges. In order to use this to improve the lower bound, we apply a first moment argument. Let $X$ be a uniformly chosen random subset of $A_1$ and denote $Y := |A_2(X)|$. Then using Jensen's inequality, we obtain
\[ \OCC{n}{d} \ge \sum_{B_1 \subset A_1} 2^{|A_2(B_1)|} = 2^{|A_1|} \cdot \E\big[2^{Y}\big] \ge 2^{|A_1| + \E[Y]} .\]
By linearity of expectation, we have $\E[Y] = (m-4)^{d-1} \cdot 2^{-(2d-2)}$ so that
\begin{equation} \label{eq:lower bound ell case}
\OCC{n}{d} \ge 2^{(m-2)^{d-1} + (m-4)^{d-1} \cdot 2^{-2d+2}} \ge 2^{\frac{n}{2d} \cdot (1-10d/m) (1 + 2^{-2d+2})} \ge 2^{\frac{n}{2d} \cdot (1 + 2^{-2d+1})}.
\end{equation}

To complete the proof of the lower bound, we now extend this bound to arbitrary (large) $n' \in 2d\N$. Let $m$ be the largest even integer satisfying $n' - 16d^2 \ge n := 2d(m^{d-1}+(d-1)m^{d-2})$.
Then, by Proposition~\ref{prop:almost-super-mult}, by~\eqref{eq:lower bound ell case}, by the maximality of $m$ and since $\OCC{2d(2d-1)}{d}=1$, we obtain
\[ \OCC{n'}{d} \ge \tfrac{1}{4d^2} \OCC{n}{d}  \ge 2^{(m-2)^{d-1} + (m-4)^{d-1} \cdot 2^{-2d+2}-\log_2 (4d^2)} \ge 2^{\frac{n'}{2d} (1 + 2^{-2d+1})} . \]

\section{The upper bound}
\label{sec:upper-bound}

For the upper bound in Theorem~\ref{thm:main}, which is the main result of this paper, we consider a slightly more general class of sets.
Recall that a subset of $\Z^d$ is called \emph{regular} if both it and its complement contain no isolated vertices.
For a graph $G$ and a positive integer $r$, we denote by $G^{\otimes r}$ the graph on the same vertex set as $G$ in which two vertices are adjacent if their distance in $G$ is at most $r$.
A finite regular subset of $\Z^d$ is an \emph{$r$-cutset} if both it and its complement are connected in $(\Z^d)^{\otimes r}$. Note that the notions of a regular cutset and a 1-cutset coincide.
We write $\OC{n,r}^d$ for the collection of odd $r$-cutsets $S$ in $\Z^d$ having $|\partial S|=n$ and $\dist(\zero,S) \le r$, where $\zero$ denotes the origin in $\Z^d$.

\begin{thm}\label{thm:strong-upper}
There exists a constant $C>0$ such that for any integers $d \ge 2$ and $n,r \ge 1$,
	\[ |\OC{n,r}^{d}|\le 2^{\frac{n}{2d} \big(1 + \frac{Cr \log^{3/2} d}{\sqrt{d}} \big)} .\]
\end{thm}

As discussed in the introduction, the proof is based on a classification of odd $r$-cutsets according to their approximate global structure. To this end, we require some definitions.
An \emph{approximation} is a pair $A=(A_\ins,A_\out)$ of disjoint subsets of $\Z^d$ such that $A_\ins$ is odd and $A_\out$ is even.
We say that $A$ approximates an odd set $S$ if $A_\ins \subset S$ and $A_\out \subset S^c$. Thus, we think of $A_\ins$ as the set of vertices known to be in $S$,
$A_\out$ as the vertices known to be outside $S$, and $A_* := (A_\ins \cup A_\out)^c$ as the
vertices whose association is unknown.

Let $1 \le t < 2d$ be an integer.
A \emph{$t$-approximation} is an approximation $A$ such that the subgraph of $\Z^d$ induced by $A_*$ has maximum degree at most $t$ and has no isolated vertices. For an illustration of these notions, see Figure~\ref{fig:approx}.
It is instructive to notice that if a \mbox{$t$-approximation} $A$ approximates $S$, then any unknown vertex is near the boundary in the sense that $A_* \subset (\intextB S)^+$; see~\eqref{eq:def-D} below.

\begin{figure}
	\captionsetup{width=0.87\textwidth}
	\captionsetup[subfigure]{justification=centering}
    \begin{subfigure}[t]{.35\textwidth}
        \includegraphics[scale=0.29]{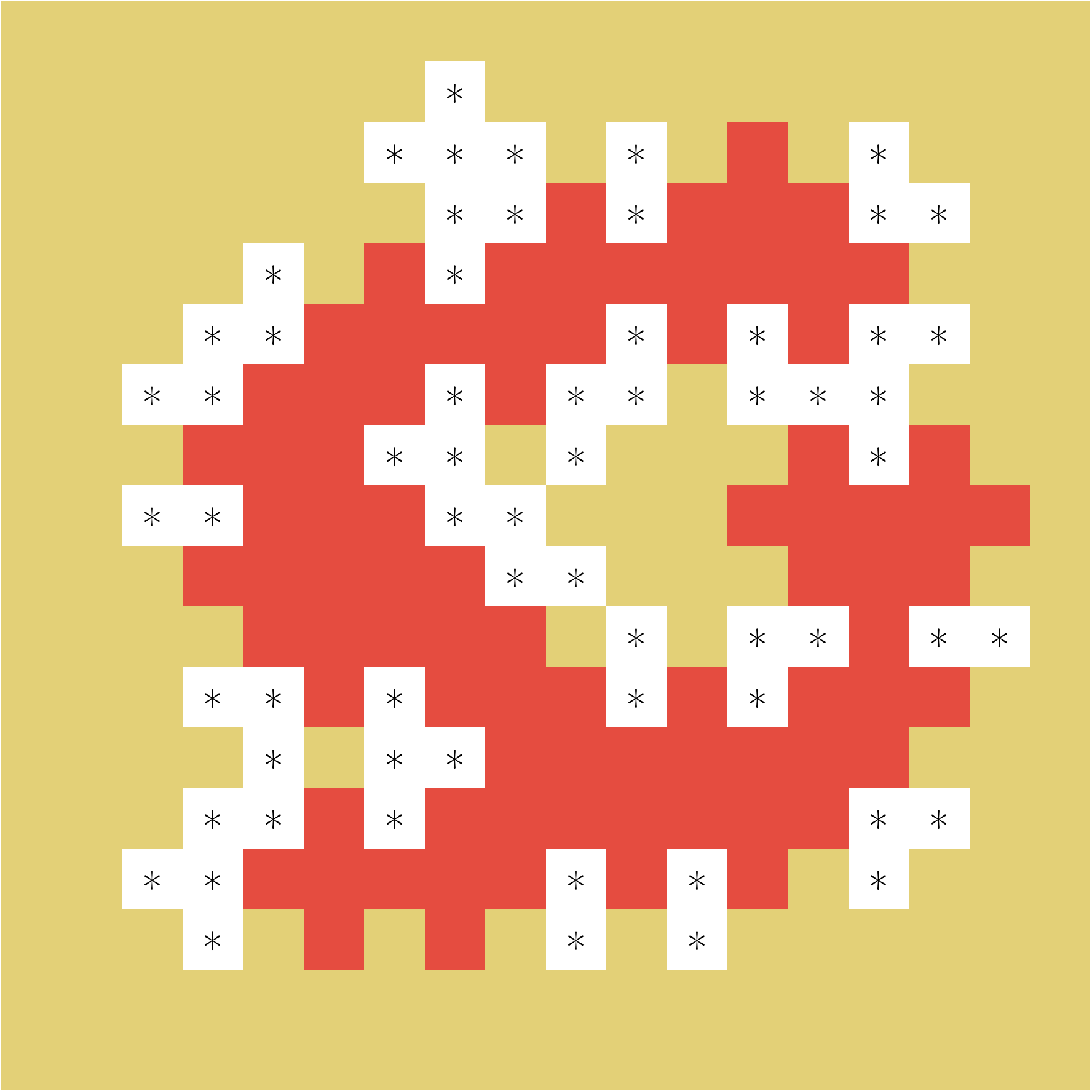}
        \caption{An approximation $A$.}
        \label{fig:approx-sample}
    \end{subfigure}%
    \begin{subfigure}{2pt}
		~
    \end{subfigure}%
    \begin{subfigure}[t]{.65\textwidth}
        \includegraphics[scale=0.29]{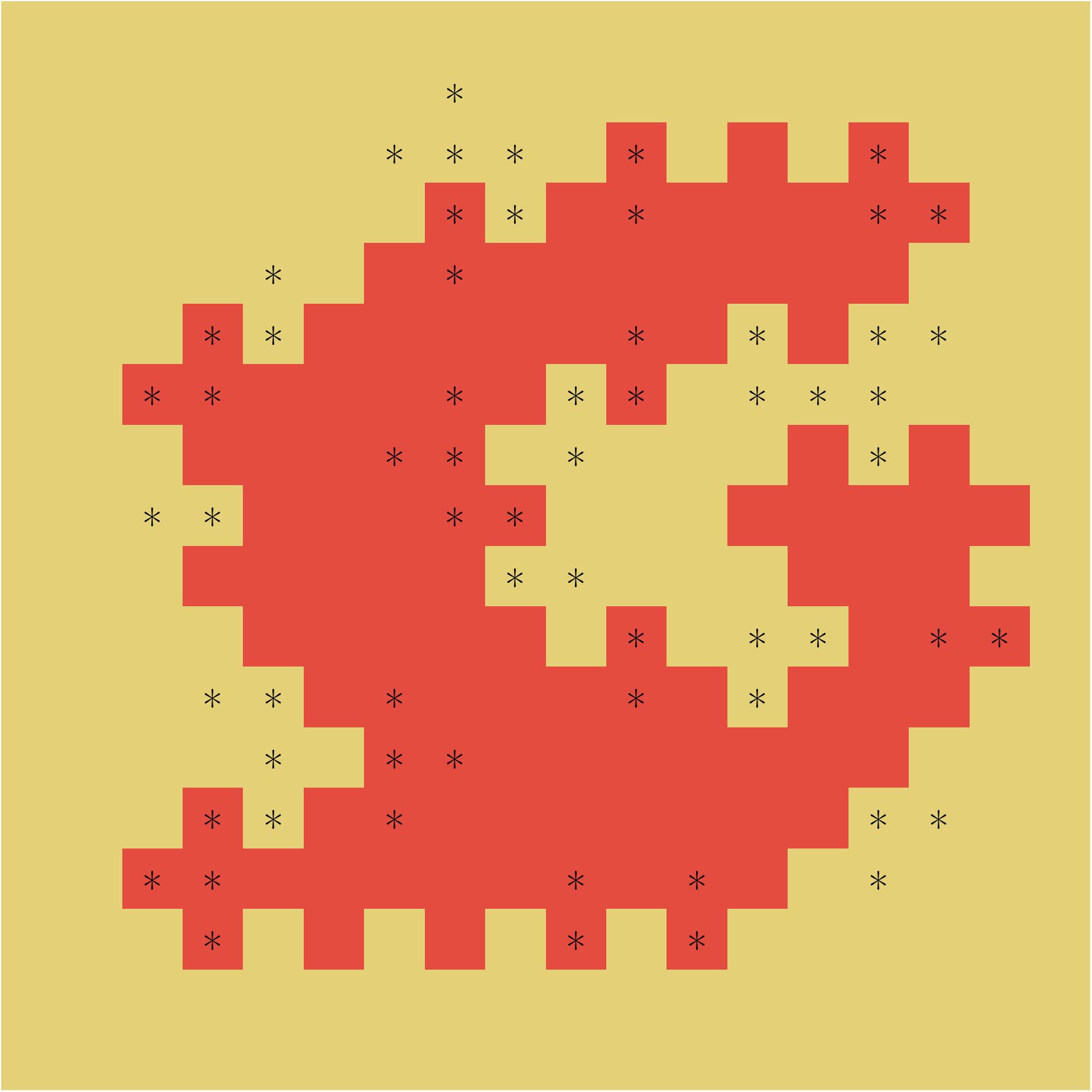}~~%
        \includegraphics[scale=0.29]{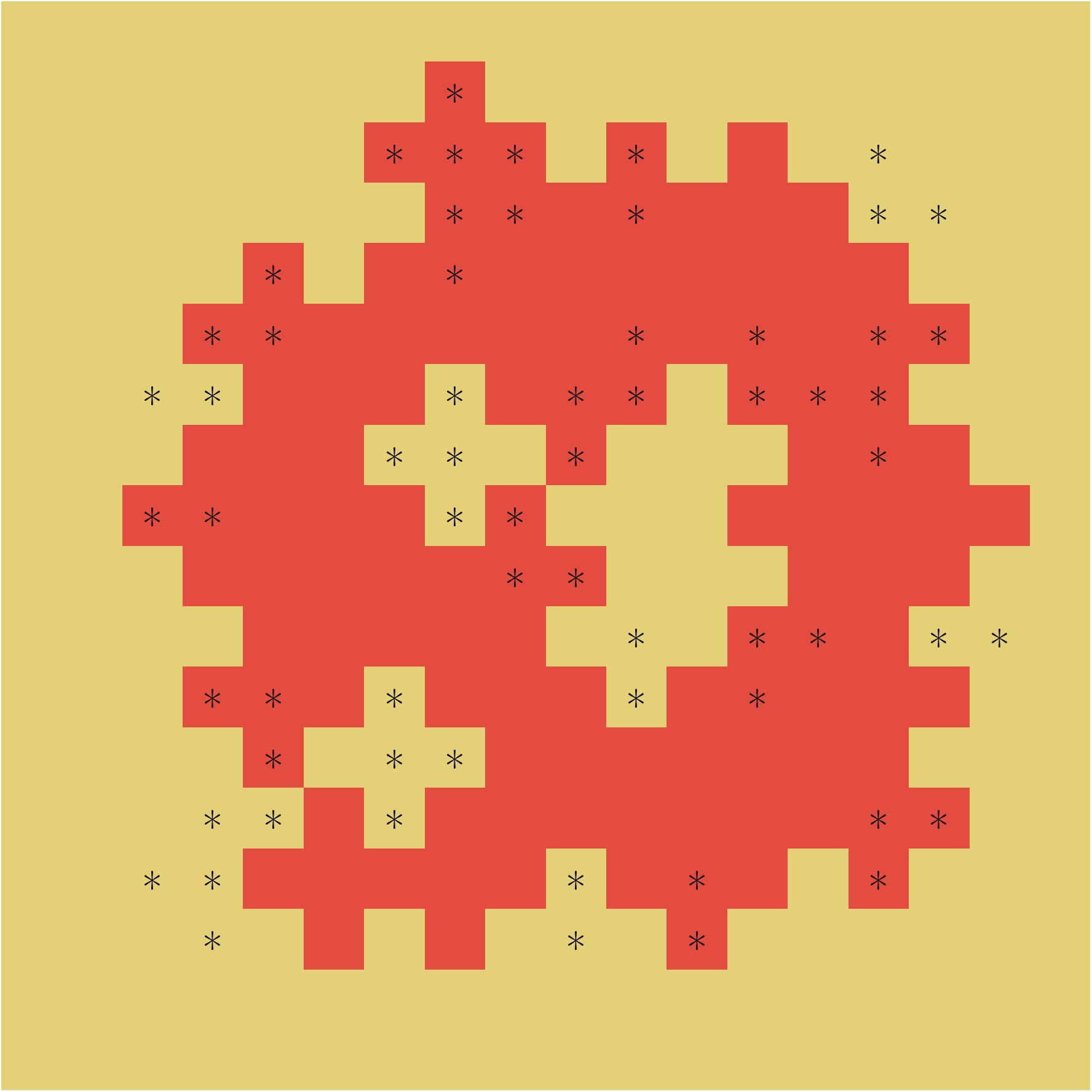}
        \caption{Two possible regular odd sets approximated by $A$.}
        \label{fig:approx-sets}
    \end{subfigure}
    \caption{An approximation and two regular odd sets approximated by it are illustrated. Vertices belonging to $A_\ins$ ($A_\out$) are known to be in $S$ ($S^c$); these are depicted in~(\textsc{\subref{fig:approx-sample}}) by a red (yellow) background. The remaining vertices belong to $A_* = (A_\ins \cup A_\out)^c$ and are unknown to be in $S$ or $S^c$; these are depicted by $*$ and a white background.}
    \label{fig:approx}
\end{figure}

We now give two key propositions which summarize the role of $t$-approximations in our proof of the upper bound.
We henceforth fix the dimension $d \ge 2$ and omit the explicit dependence on $d$ in the notation.
Denote by $\cut$ the collection of regular odd sets and by $\cut_n$ the collection of $S \in \cut$ having $|\partial S| = n$.
For an approximation $A$, denote by $\cut(A)$ the collection of $S \in \cut$ which are approximated by $A$.
We extend this notation to a family of approximations $\cA$, by setting
$\cut(\cA) := \cup_{A\in \cA} \cut(A)$.
Our first proposition justifies our notions of approximation by bounding the number of regular odd sets approximated by a given $t$-approximation.

\begin{prop}\label{lem:number-of-cutsets-with-approx}
	For any integers $n \ge 1$ and $1 \le t < 2d$ and any $t$-approximation $A$, we have
	\[ |\cut_n(A)| \le 2^{n/(2d-t)} .\]
\end{prop}

Our second proposition shows that a small family of $t$-approximations suffices to approximate every set in $\OC{n,r}$.

\begin{prop}\label{lem:family-of-approx}
	There exists a constant $C>0$ such that for any integers $n,r \ge 1$ and $1 \le t < 2d$, there exists a family $\cA$ of $t$-approximations of size
	\[ |\cA| \le \exp\left(Cnr \big(\tfrac{\log d}{d}\big)^{3/2} + \tfrac{Cn \log d}{dt} \right) \]
	such that every $S \in \OC{n,r}$ is approximated by some element in $\cA$, i.e., $\OC{n,r} \subset \cut(\cA)$.
\end{prop}

The proofs of Proposition~\ref{lem:number-of-cutsets-with-approx} and Proposition~\ref{lem:family-of-approx} are given in Section~\ref{sec:number-of-cutsets-with-approx} and Section~\ref{sec:approx}, respectively. Equipped with these propositions, we are now ready to prove the upper bound.

\begin{proof}[Proof of upper bound in Theorem~\ref{thm:strong-upper}]
	Let $d \ge 2$, $1 \le t < 2d$ and $n,r \ge 1$ be integers. Let $\cA$ be a family of $t$-approximations obtained by applying Proposition~\ref{lem:family-of-approx}.
	By Proposition~\ref{lem:number-of-cutsets-with-approx},
	\[
	|\OC{n,r}| \le \sum_{A \in \cA} |\cut_n(A)| \le |\cA| \cdot 2^{n/(2d-t)} \le 2^{\tfrac{n}{2d}\Big(1 + \tfrac{t}{2d-t} + \tfrac{Cr \log^{3/2} d}{\sqrt{d}} + \tfrac{C \log d}{t} \Big)}.
	\]
	Substituting any integer $t$ satisfying $\sqrt{d / \log d} \le t \le \sqrt{d} \log^{3/2} d$ yields the theorem.
\end{proof}

%
\section{Counting regular odd sets with a given approximation}
\label{sec:number-of-cutsets-with-approx}
%

In this section, we prove Proposition~\ref{lem:number-of-cutsets-with-approx}. That is, our goal is to prove an upper bound on the number of regular odd sets with given boundary size which are approximated by a particular $t$-approximation.
The proof is based on an analysis of minimal vertex-covers.

We henceforth fix an integer $1 \le t < 2d$ and a $t$-approximation $A=(A_\ins,A_\out)$.
Recall our notation $A_* := (A_\ins \cup A_\out)^c$.
For $S \in \cut(A)$, define
\begin{equation}\label{eq:def-D}
\begin{aligned}
D_{\ins}(S) &:= A_* \cap \intB S = \Odd \cap A_* \cap S ,\\
D_{\out}(S) &:= A_* \cap \extB S = \Even \cap A_* \cap S^c ,
\end{aligned}
\end{equation}
where the first equality follows from $\Odd \cap A_* \cap S \subset N(A_\out) \subset N(S^c)$, which in turn uses the facts that $A_\ins$ is odd and the maximum degree of $A_*$ is strictly less than $2d$; the second equality follows similarly.
Define also
\[ D(S) := D_{\ins}(S) \cup D_{\out}(S) = A_* \cap \intextB S .\]
Two key properties of this definition are that $S$ is determined by $D(S)$ and that $D(S)$ is a minimal vertex-cover of $A_*$ (see Figure~\reffig{fig:recovery}). This is stated precisely in the following lemma.

A {\em vertex-cover} of a graph $G$ is a subset of vertices $U \subset V(G)$ satisfying that every edge of $G$ has an endpoint in $U$.
A vertex-cover is \emph{minimal} if it is minimal with respect to inclusion.
Denote by $\MC(G)$ the set of all minimal vertex-covers of $G$.
For a set $V \subset \Z^d$, we also write $\MC(V)$ for the set of minimal covers of the subgraph of $\Z^d$ induced by $V$.

\begin{lemma}\label{lem:D-properties}
	The map $S \mapsto D(S)$ is an injective map from $\cut(A)$ to $\MC(A_*)$
\end{lemma}
\begin{proof}
	Let $S \in \cut(A)$ and denote $D_\ins := D_\ins(S)$, $D_\out := D_\out(S)$ and $D := D_\ins \cup D_\out$.
	To see that the map is injective, it suffices to reconstruct $S$ from $D$. In fact, we can reconstruct $S$ both from $D_\ins$ and from $D_\out$, separately.
Indeed, as $A_\ins \subset S$ and $A_\out \subset S^c$, it follows that
\[ \Odd \cap S = \Odd \cap A_\ins \cup D_\ins \quad\text{and}\quad \Even \cap S^c = \Even \cap  A_\out \cup D_\out ,\]
and since $S$ is regular odd, $S$ is determined by $\Odd \cap S$ via $S = (\Odd \cap S) \cup N_{2d}(\Odd \cap S)$ and by $\Even \cap S^c$ via $S = (\Even \cap S)^+ = (\Even \setminus (\Even \cap S^c))^+$.

	Next, we show that $D$ is a vertex-cover of $A_*$.
	To this end, let $u,v \in A_*$ be a pair of adjacent vertices, and assume without loss of generality that $u$ is odd and $v$ is even.
	Assume towards obtaining a contradiction that neither $u$ nor $v$ belong to $D$, and observe that in this case $u \notin S$ and $v \in S$, which is impossible since $S$ is odd. Hence either $u\in D$ or $v\in D$.
	
	Finally, we show that $D$ is a minimal vertex-cover.
	To this end, let $v \in D$ and assume towards a contradiction that $N(v) \cap A_* \subset D$. Assume without loss of generality that $v$ is odd, so that $v \in D_\ins$ and $N(v) \cap A_* \subset D_\out$. Since $A_\ins$ is odd, $v$ is odd and $v \notin A_\ins$, we have $N(v) \cap A_\ins = \emptyset$. Thus, $N(v) \subset A_\out \cup D_\out \subset S^c$, which is impossible since $v \in S$ and $S$ is regular.
\end{proof}

\begin{figure}
	\centering
	\begin{subfigure}[t]{.32\textwidth}
		\centering
		\includegraphics[scale=0.4]{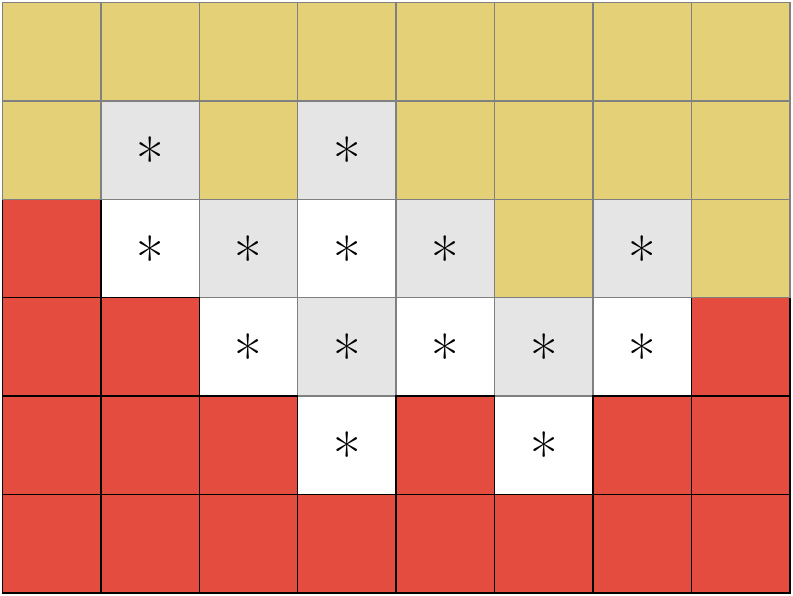}
		\caption{A region of unknown vertices in $A_*$. These vertices are denoted by $*$ (with odd vertices having a gray background). The vertices in $A_\ins$ and $A_\out$ are shown in red and yellow, respectively.}
		\label{fig:recovery-1}
	\end{subfigure}%
	\begin{subfigure}{20pt}
		\quad
	\end{subfigure}%
	\begin{subfigure}[t]{.62\textwidth}
		\centering
		\includegraphics[scale=0.4]{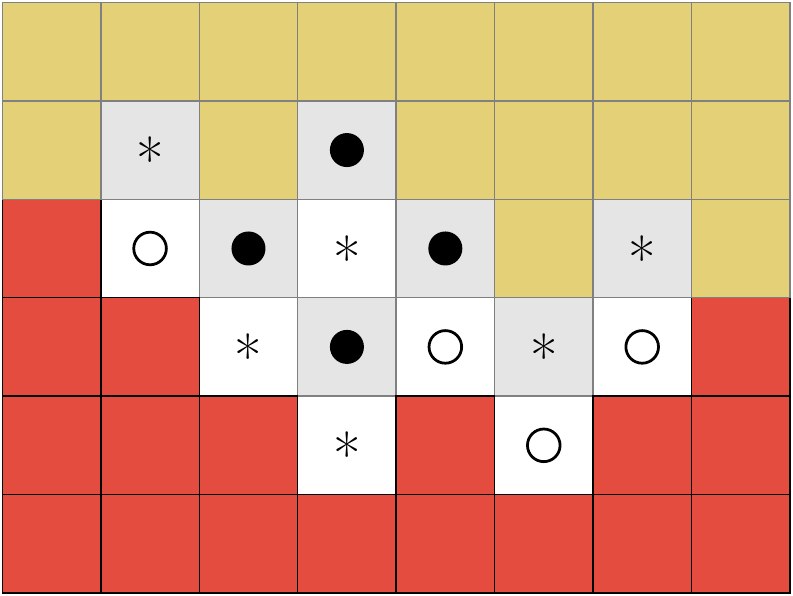}\qquad
		\includegraphics[scale=0.4]{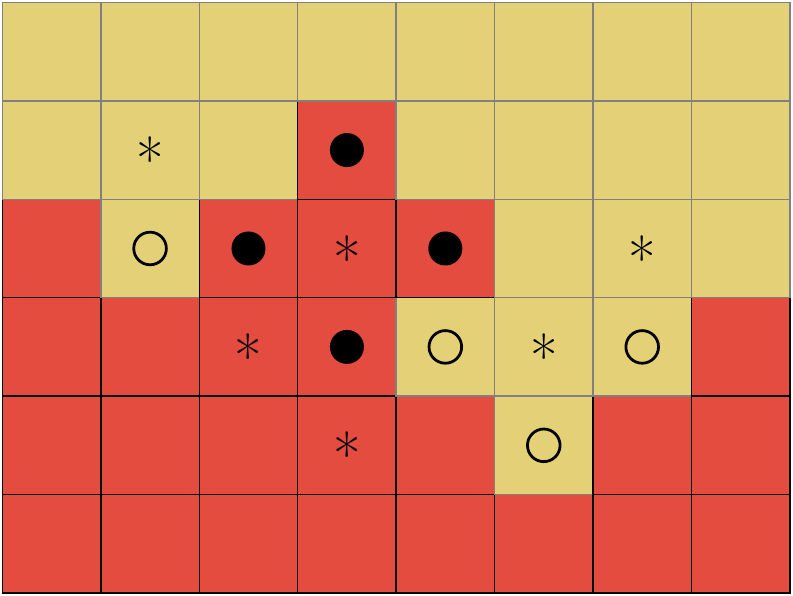}
		\caption{An example of $(D_{\ins},D_{\out})$ and its corresponding regular odd set. The property that $D_{\ins} \cup D_{\out}$ is a minimal vertex-cover of $A_*$ is manifested in the figure by the fact that there are no two adjacent $*$. The corresponding regular odd set is obtained by adding each vertex in $D_{\ins}$ to $S$ and each vertex in $D_{\out}$ to $S^c$, and then determining the remaining vertices according to their neighbors.}
		\label{fig:recovery-2}
	\end{subfigure}
	\caption{The figure illustrates the process of recovering $S$ from $D_{\ins}$ and $D_{\out}$.}
	\label{fig:recovery}
\end{figure}

We require the following lemma from~\cite{feldheimspinka}.
\begin{lemma}[{\cite[Lemma~4.9]{feldheimspinka}}]\label{lem:sum-over-minimal-covers}
	Let $G$ be a finite graph and let $\{p_v\}_{v \in V(G)}$ be non-negative numbers satisfying $p_u + p_v \le 1$ for all $\{u,v\} \in E(G)$.
	Then
	\[ \sum_{U \in \MC(G)} \prod_{u \in U} p_u \le 1 .\]
\end{lemma}

Applying this with $p_v=1/2$ for all $v \in V(G)$, yields
\[ |\mathcal{U}| \le \max_{U \in \mathcal{U}} 2^{|U|} , \quad\text{for any } \mathcal{U} \subset \MC(G) .\]
Hence, Lemma~\ref{lem:D-properties} implies that for any $n \ge 1$,
\[ |\cut_n(A)| \le |\{ D(S) : S \in \cut_n(A) \}| \le \max_{S \in \cut_n(A)} 2^{|D(S)|} .\]
Proposition~\ref{lem:number-of-cutsets-with-approx} is now an immediate consequence of
 the following lemma.

\begin{lemma}\label{lem:tightness}
For any $S \in \cut(A)$, we have
\[ |D(S)| \le \frac{|\partial S|}{2d-t} .\]
\end{lemma}

\begin{proof}
Let $S \in \cut(A)$ and denote $D_\ins := D_\ins(S)$, $D_\out := D_\out(S)$ and $D := D_\ins \cup D_\out$.
Since $A_\ins$ is odd, $A_\out$ is even and $A_*$ induces a subgraph of maximum degree at most $t$, we have
\[ \Odd \cap A_* \subset N_{2d-t}(A_\out) \quad\text{and}\quad \Even \cap A_* \subset N_{2d-t}(A_\ins) .\]
Thus,
\[ |D_\ins| \le \tfrac{|\partial(D_\ins, A_\out)|}{2d-t} \quad\text{and}\quad |D_\out| \le \tfrac{|\partial(D_\out, A_\ins)|}{2d-t} . \]
Since $\partial(D_\ins, A_\out)$ and $\partial(D_\out, A_\ins)$ are disjoint subsets of $\partial S$ (since $A_\ins \subset S$ and $A_\out \subset S^c$), we have
\[ |D| = |D_\ins \cup D_\out| \le \frac{|\partial S|}{2d-t} . \qedhere \]
\end{proof}

%
\section{Constructing approximations}
\label{sec:approx}
%

This section is dedicated to the proof of Proposition~\ref{lem:family-of-approx}.
That is, we show that there exists a small family $\cA$ of $t$-approximations which covers $\OC{n,r}$ in the sense that $\OC{n,r}\subset \cut(\cA)$.
The construction of $\cA$ is done in two steps, which we outline here.
For an approximation $A$, recall the notation
$A_*:=(A_\ins\cup A_\out)^c$ and say that $|A_*|$ is the \emph{size} of $A$.
The first step is to construct a small family of small approximations which covers $\cut_{n,r}$.

\begin{lemma}\label{lem:family-of-rough-approx}
For any integers $n \ge 1$ and $r \ge 1$, there exists a family $\cA$ of approximations, each of size at most $Cn \sqrt{(\log d)/d}$,
 such that $\OC{n,r} \subset \cut(\cA)$ and
	\[ |\cA| \le \exp\Big(\tfrac{Cn r \log^{3/2} d}{d^{3/2}}\Big) .\]
\end{lemma}

The second step is to upgrade an approximation to a small family of $t$-approximations which covers at least the same collection of regular odd sets.

\begin{lemma}\label{lem:family-of-full-approx}
For any integers $n,m \ge 1$ and $1 \le t < 2d$ and any approximation $A$ of size $m$, there exists a family $\cA$ of $t$-approximations such that $\OC{n}(A) \subset \cut(\cA)$ and
	\[ |\cA| \le \exp\Big(\tfrac{C\log d}{d} \cdot \big(m + \tfrac{n}{t}\big)\Big) .\]
\end{lemma}

Lemma~\ref{lem:family-of-rough-approx} and Lemma~\ref{lem:family-of-full-approx} are proved in Sections~\ref{sec:small-approx}, and~\ref{sec:full-approx} below.

\begin{proof}[Proof of Proposition~\ref{lem:family-of-approx}]
	Applying Lemma~\ref{lem:family-of-rough-approx}, we obtain a family $\cB$ of approximations, each of size at most $m := Cn \sqrt{(\log d)/d}$, such that $\cut_{n,r} \subset \cut(\cB)$ and $|\cB| \le \exp( Cnr d^{-3/2} \log^{3/2} d)$.
	Applying Lemma~\ref{lem:family-of-full-approx} to each approximation in $\cB$, we obtain a collection of families of $t$-approximations.
	Taking the union over this collection, we obtain a family $\cA$ of $t$-approximations such that $\OC{n,r}~\subset~\cut(\cA)$ and
	$|\cA| \le |\cB| \cdot \exp(Cnd^{-1} \log d \cdot (\sqrt{(\log d)/d} + 1/t))$.
	The required bound follows.
\end{proof}

\subsection{Preliminaries}

In this section, we gather some elementary combinatorial facts about graphs which we require for the construction of approximations.
For the purpose of these preliminaries, we fix an arbitrary graph $G=(V,E)$ of maximum degree $\Delta$.

\begin{lemma}\label{lem:sizes}
	Let $U \subset V$ be finite and let $t>0$. Then
	\[ |N_t(U)| \le \frac{\Delta}{t} \cdot |U| .\]
\end{lemma}
\begin{proof}
	This follows from a simple double counting argument.
	\[ t |N_t(U)|
		\le \sum_{v \in N_t(U)} |N(v) \cap U|
		= \sum_{u \in U} \sum_{v \in N_t(U)} \1_{N(u)}(v)
		= \sum_{u \in U} |N(u) \cap N_t(U)|
		\le \Delta |U| . \qedhere \]
\end{proof}

The next lemma follows from a classical result of Lov{\'a}sz~\cite[Corollary~2]{lovasz1975ratio} about fractional vertex covers,
applied to a weight function assigning a weight of $\frac1t$ to each vertex of $S$.

\begin{lemma}\label{lem:existence-of-covering2}
Let $S \subset V$ be finite and let $t \ge 1$. Then there exists a set $T \subset S$ of size $|T|~\hspace{-4pt}\le~\hspace{-4pt}\frac{1+\log \Delta}{t} |S|$ such that $N_t(S) \subset N(T)$.
\end{lemma}
%

The following standard lemma gives a bound on the number of connected subsets of a graph.
\begin{lemma}[{\cite[Chapter~45]{Bol06}}]\label{lem:number-of-connected-graphs}
The number of connected subsets of $V$ of size $k+1$ which contain the origin is at most $(e(\Delta-1))^k$.
\end{lemma}

Recall that $G^{\otimes r}$ is the graph on $V$ in which two vertices are adjacent if their distance in $G$ is at most $r$.
The next simple lemma was first introduced by Sapozhenko~\cite{sapozhenko1987onthen}.

\begin{lemma}[{\cite[Lemma 2.1]{sapozhenko1987onthen}}]\label{lem:r-connected-sets}
Let $S,T \subset V$ and let $a,b$ be positive integers.
Assume that $S$ is connected in $G^{\otimes a}$, $\dist(s,T) \le b$ for all $s \in S$ and $\dist(S,t) \le b$ for all $t \in T$. Then $T$ is connected in $G^{\otimes (a+2b)}$.
\end{lemma}

The following lemma, based on ideas of Tim{\'a}r~\cite{timar2013boundary}, establishes the connectivity of the boundary of subsets of $\Z^d$ which are both connected and co-connected.

\begin{lemma}[{\cite[Proposition~3.1]{feldheim2013rigidity}}] \label{lem:int+ext-boundary-is-connected}
Let $U \subset \Z^d$ be connected and co-connected. Then $\intextB U$ is connected.
\end{lemma}

\begin{cor}\label{lem:int+ext-boundary-is-r-connected}
	Let $r \ge 1$ be an integer and let $U \subset \Z^d$ be such that $U$ and $U^c$ are connected in $(\Z^d)^{\otimes r}$. Then $\intextB U$ is connected in $(\Z^d)^{\otimes r}$.
\end{cor}
\begin{proof}
	Since $U$ is connected in $(\Z^d)^{\otimes r}$, it suffices to show that $\intextB B$ is connected in $(\Z^d)^{\otimes r}$ for every connected component $B$ of $U$.
	Let $\cC$ be the collection of connected components of $B^c$.
	Since $U^c$ is connected in $(\Z^d)^{\otimes r}$, it suffices to show that $\intextB W \cup \intextB W'$ is connected in $(\Z^d)^{\otimes r}$ whenever $W,W' \in \cC$ satisfy $\dist(W,W') \le r$. This follows from Lemma~\ref{lem:int+ext-boundary-is-connected}.
\end{proof}

\subsection{Constructing small approximations}
\label{sec:small-approx}

This section is devoted to the proof of Lemma~\ref{lem:family-of-rough-approx}.
That is, we construct a small family of approximations, each of size at most $Cn \sqrt{(\log d)/d}$, such that $\OC{n,r} \subset \cut(\cA)$. This is done in two steps. First, we show that for every regular odd set $S$, there exists a small set $U$ such that $N(U)$ separates $S$, where we say that a set $W$ \emph{separates} $S$ if every edge in $\partial S$ has an endpoint in $W$.

\begin{lemma}\label{lem:existence-of-U}
	Let $n \ge 1$ be an integer and let $S \in \cut_n$. Then there exists $U \subset (\intextB S)^+$
of size at most $Cn d^{-3/2}\sqrt{\log d}$ such that $N(U)$ separates $S$.
\end{lemma}

We then show that every separating set gives rise to a small family
of small approximations.

\begin{lemma}\label{lem:family-of-small-approx}
	For any integer $n \ge 1$ and any finite $W\subset \Z^d$,
there exists a family $\cA$ of approximations, each of size at most $3|W|$, such that every $S\in \cut_n$ which is separated by $W$ satisfies $S\in \cut(\cA)$, and $|\cA| \le 4^{|W|/d}$.
\end{lemma}

Before proving these lemmas, let us show how they imply Lemma~\ref{lem:family-of-rough-approx}.

\begin{proof}[Proof of Lemma~\ref{lem:family-of-rough-approx}]
Let $n,r \ge 1$ be integers.
	By Corollary~\ref{cor:odd-set-min-boundary-size}, if $\cut_n$ is non-empty then $n \ge d^2$. Thus, we may assume that $n \ge d^2$.
	Denote $k:=C n d^{-3/2} \sqrt{\log d}$, $V:=\{\zero+ie_1 : 0 \le i < n^2 \}$ and $\widebar{V} := \{ v \in \Z^d : \dist(v,V) \le r+2 \}$. Let $\cU$ be the collection of all subsets of $\Z^d$ of size at most $k$ which intersect $\widebar{V}$ and are connected in $(\Z^d)^{\otimes (r+4)}$.
Since the maximum degree of $(\Z^d)^{\otimes (r+4)}$ is at most $(Cd)^{r+4}$, Lemma~\ref{lem:number-of-connected-graphs} implies that
\[ |\cU| \le |\widebar{V}| \cdot (e(Cd)^{r+4})^k \le \exp\left(Cnr d^{-3/2} \log^{3/2} d \right) ,\]
where the rightmost inequality uses the fact that $|\widebar{V}| \le n^2 (2d+1)^{r+2}$ and $n \ge d^2$.

For each $U \in \cU$, apply Lemma~\ref{lem:family-of-small-approx} to $W=N(U)$ to obtain a family $\cA_U$ of approximations, each of size at most $3|N(U)| \le 6dk$, such that every $S \in \cut_n$ which is separated by $N(U)$ satisfies $S \in \cut(\cA_U)$, and $|\cA_U| \le 4^{2k}$. Denote $\cA := \cup_{U \in \cU} \cA_U$ and note that $\cA$ is a family of approximations, each of size at most $6dk$, such that
\[ |\cA| \le |\cU| \cdot 4^{2k} \le \exp\left(Cnr d^{-3/2} \log^{3/2} d \right) .\]

It remains to check that $\cut_{n,r} \subset \cut(\cA)$.
Towards showing this, let $S \in \cut_{n,r}$.
By Lemma~\ref{lem:existence-of-U}, there exists $U \subset (\intextB S)^+$ of size at most $k$ such that $N(U)$ separates $S$. Thus, since $S \in \cut(\cA_U)$ by definition, to conclude that $S \in \cut(\cA)$, it suffices to show that $U \in \cU$. Since $|U| \le k$, we need only show that $U$ is connected in $(\Z^d)^{\otimes (r+4)}$ and that $U$ intersects $\widebar{V}$.

We first show that $U$ intersects $\widebar{V}$, or equivalently, that $\dist(U,V) \le r+2$.
Since $N(U)$ separates $S$, we have $\intextB S \subset N(U)^+$ so that it suffices to show that $\dist(\intextB S,V) \le r$.
Indeed, if $\zero \notin S$ then $\dist(\zero,\intextB S) \le r$, since $\dist(\zero,S) \le r$, and if $\zero \in S$ then $V \cap \intextB S \neq \emptyset$, since, by Lemma~\ref{lem:isoperimetry}, $|S| \le |\partial S|^2 \le n^2$.

We are left with showing that $U$ is connected in $(\Z^d)^{\otimes (r+4)}$.
Indeed, since $U \subset (\intextB S)^+$, we see that $\dist(u,\intextB S) \le 1$ for all $u \in U$, and since $\intextB S \subset N(U)^+$, we have $\dist(w,U) \le 2$ for all $w \in \intextB S$. As $S$ and $S^c$ are connected in $(\Z^d)^{\otimes r}$, Corollary~\ref{lem:int+ext-boundary-is-r-connected} and Lemma~\ref{lem:r-connected-sets} imply that $U$ is connected in $(\Z^d)^{\otimes (r+4)}$.
\end{proof}

\subsubsection{Constructing separating sets}
Before proving Lemma~\ref{lem:existence-of-U}, we start with a basic geometric property of odd sets which we require for the construction of the separating set.

\begin{lemma}\label{lem:four-cycle-property}
	Let $S$ be an odd set and let $\{u,v\} \in \partial S$. Then, for any unit vector $e \in \Z^d$, either $\{u,u+e\}$ or $\{v,v+e\}$ belongs to $\partial S$. In particular,
	\[ |\partial u \cap \partial S| + |\partial v \cap \partial S| \ge 2d .\]
\end{lemma}

\begin{proof}
	Assume without loss of generality that $u$ is odd. Since $S$ is odd, we have $u \in S$ and $v \notin S$.
	Similarly, if $u+e \in S$ then $v+e \in S$. Thus, either $\{u,u+e\} \in \partial S$ or $\{v,v+e\} \in \partial S$.
\end{proof}

For a set $S$, denote the {\em revealed vertices} in $S$ by
\[ S^{\rev} := \{ v \in \Z^d ~:~ |\partial v \cap \partial S| \geq d \} .\]
That is, a vertex is revealed if it sees the boundary in at least half of the $2d$ directions.
The following is an immediate corollary of Lemma~\ref{lem:four-cycle-property}.

\begin{cor}\label{cor:revealed-separate}
	Let $S$ be an odd set. Then $S^{\rev}$ separates $S$.
\end{cor}

\begin{figure}
	\centering
	\includegraphics[scale=1.2]{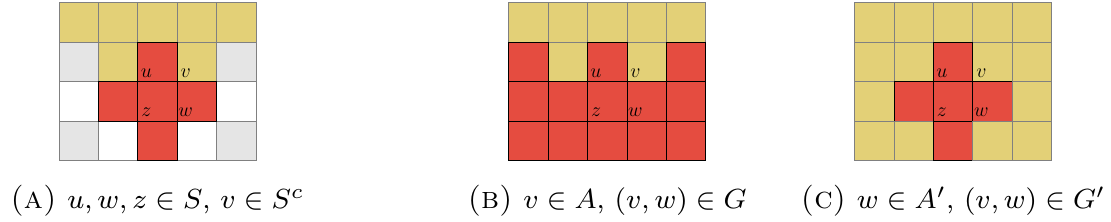}
	\caption{Constructing the separating set. In~\textsc{(a)}, a revealed vertex $u \in S$ is depicted along with a neighbor $z \in S$. Every four-cycle $(u,v,w,z)$ such that $v \in S^c$ (and hence $w \in S$) falls into one of two types. Either $v$ has at least $s$ boundary edges as shown in~\textsc{(b)}, or $w$ has at least $2d-s$ boundary edges as in~\textsc{(c)}. At least $1/2$ of all such four-cycles belong to the same type. If it is the first type, then $u$ is adjacent to many vertices which have many boundary edges and one such neighbor of $u$ is included in $B$; if it the second type, then $z$ is adjacent to many vertices which have almost all their edges in the boundary and $z$ is included in $B'$. The set $U_S$ is obtained by taking the union of $B$ and $B'$.}
	\label{fig:existence-of-U}
\end{figure}

\begin{proof}[Proof of Lemma~\ref{lem:existence-of-U}]
	Let $n \ge 1$ and let $S \in \cut_n$.
	Note that $\partial S = \partial S^c$ implies that $S^{\rev} = (S^c)^{\rev}$. Thus, in light of Corollary~\ref{cor:revealed-separate}, it suffices to show that, for each $R\in\{S,S^c\}$, there exists a set $U_R \subset N(\intB R)$ such that $R \cap R^{\rev} \subset N(U_R)$ and $|U_R| \le Cn d^{-3/2} \sqrt{\log d}$.
	Indeed, the lemma then follows by taking $U := U_S \cup U_{S^c}$.
	Since $S$ and $S^c$ are symmetric (up to parity), we may consider the case $R=S$.
	The proof is accompanied by Figure~\ref{fig:existence-of-U}.
	
	Denote $s:=\sqrt{d \log d}$ and $t:=d/4$, and define
	\[ A := \extB S \cap N_s(\intB S) \quad\text{and}\quad A' := \intB S \cap N_{2d-s}(\extB S). \]
	Observe that, by Lemma~\ref{lem:sizes},
	\[ |A| \le \frac{n}{s} \quad\text{ and }\quad |A'| \le \frac{n}{2d-s}. \]
	We now use Lemma~\ref{lem:existence-of-covering2} with $A$ to obtain a set $B \subset A \subset \extB S$ such that
	\[ |B|\le\frac{3\log d}t|A| \quad\text{and}\quad N_t(A) \subset N(B). \]
	We also define $B':= S \cap N_t(A')$. By Lemma~\ref{lem:sizes}, we have
	\[ |B'| \le \frac{2s}{t}|A'| .\]
	Finally, we define $U_S:=B \cup B'$. Clearly, $U_S \subset N(\intB S)$ and
	\[ |U_S| \le \frac{3n \log d}{ts} +\frac{2sn}{t(2d-s)} \le \frac{12n \log d}{d\sqrt{d\log d}} +\frac{8n\sqrt{d\log d}}{d^2} =
	\frac{20n \sqrt{\log d}}{d^{3/2}} .\]
	
	It remains to show that $S \cap S^{\rev} \subset N(U_S)$.
	Towards showing this, let $u \in S \cap S^{\rev} = \intB S \cap N_d(\extB S)$.
	Since $S$ is regular, there exists a vertex $z \in N(u) \cap S$.
	Let $F$ denote the set of pairs $(v,w)$ such that $(u,v,w,z)$ is a four-cycle and $v \in \extB S$, and note that $|F| \ge d-1$.
	Denote
	\[ G := \{ (v,w) \in F ~:~ v \in A \} \quad\text{and}\quad G' := \{ (v,w) \in F ~:~ w \in A' \} .\]
	Note that, by Lemma~\ref{lem:four-cycle-property}, $F = G \cup G'$ and, for any $(v,w) \in F$, we have $w \in S$.
	Since $F = G \cup G'$, either $|G|$ or $|G'|$ is at least $|F|/2 \ge t$.	
	Now observe that
	if $|G| \ge t$ then $u \in N_t(A) \subset N(B)$, while if $|G'| \ge t$ then $z \in N_t(A')$ so that $u \in N(B')$.
	Therefore, we have shown that $u \in N(U_S)$.
\end{proof}

\subsubsection{From separating sets to small approximations}

\begin{proof}[Proof of Lemma~\ref{lem:family-of-small-approx}]
Let $n \ge 1$ and let $W \subset \Z^d$.
Consider the set $X := \Z^d \setminus W$.
Say that a connected component of $X$ is {\em small} if its
size is at most $d$, and that it is {\em large} otherwise.

Let $S \in \cut_n$ be such that $W$ separates $S$ and observe that every connected component of $X$ is entirely contained in either $S$ or $S^c$.
Thus, if we let $B_\ins$ and $B_\out$ be the union of all the large components of $X$ which are contained in $S$ and $S^c$, respectively, we have that $B_\ins \subset S$ and $B_\out \subset S^c$.
To obtain an approximation of $S$ from $(B_\ins,B_\out)$, define $A_\ins := B_\ins \cup (\Odd \cap B_\ins^+)$ and
$A_\out := B_\out \cup (\Even \cap B_\out^+)$. Clearly, $A_\ins$ is odd and $A_\out$ is even, and, since $S$ is odd, $A_\ins \subset S$ and $A_\out \subset S^c$. Hence, $A = A(S) := (A_\ins, A_\out)$ is an approximation of $S$, i.e., $S \in \cut(A)$.

Next, we bound the size of $A$. For this we require a simple corollary of Lemma~\ref{lem:isoperimetry}. Namely,
\[ |\partial T| \ge d \cdot \min\{d,|T|\} \qquad\text{ for any finite }T \subset \Z^d .\]
Indeed, this follows immediately from Lemma~\ref{lem:isoperimetry} since $2 \ge e^{1/e} \ge x^{1/x}$ for all $x>0$.
Denoting by $Y$ the union of all the small
components of $X$, we observe that $A_* \subset Y \cup W$.
Since any small component $T$ of $X$ has $|T| \le |\partial T|/d$ and $\partial T \subset \partial W$, we obtain
\[ |Y| \le \frac{|\partial W|}{d} \le \frac{2d |W|}{d} \le 2 |W| .\]
Thus, $|A_*| \le |Y \cup W| \le 3 |W|$.

Now, denote by $\cA$ the collection of approximations $A(S)$ constructed above for all $S\in \cut_n$ which are
separated by $W$. To conclude the proof, it remains to bound $|\cA|$.
Let $\ell$ be the number of large components of $X$,
and observe that $|\cA| \le 2^{\ell}$.
Since any large component $T$ of $X$ has
$|\partial T| \ge d^2$ and $\partial T \subset \partial W$, we obtain
$\ell \le |\partial W|/d^2 \le 2|W|/d$ so that $|\cA| \le 4^{|W|/d}$, as required.
\end{proof}

\subsection{Constructing \texorpdfstring{$t$}{t}-approximations}
\label{sec:full-approx}

In this section, we take a small approximation and refine it into a multitude of
$t$-approximations.
The following lemma allows us to eliminate any isolated unknown vertices in an approximation.
\begin{lemma}\label{lem:eliminate isolated}
For every approximation $A$ there exists an approximation $B$ such that $\cut(A)=\cut(B)$, $B_* \subset A_*$ and $B_*$ has no isolated vertices.
\end{lemma}
\begin{proof}
The lemma clearly holds when $\cut(A)=\emptyset$ so that we may assume that $\cut(A)\neq \emptyset$.
Define $B_\ins := A_\ins \cup N_{2d}(A_\ins)$ and $B_\out := A_\out \cup N_{2d}(A_\out)$.
Using the definition of a regular set and the assumption that $\cut(A)$ is non-empty, it is straightforward to check that $B=(B_\ins,B_\out)$ is an approximation and that $\cut(B) = \cut(A)$.
Finally, since $A_\ins$ is odd and $A_\out$ is even, we have that the set of isolated vertices of $A_*$ is
$A_*\setminus {N(A_*)}=N_{2d}(A_\out)\cup N_{2d}(A_\ins)$, and similarly for $B_*$.
Thus, $B_*$ has no isolated vertices.
\end{proof}

For an approximation $A$ and an integer $m \ge 0$, we define
\[ \OC{m}^*(A) := \big\{ S \in \cut(A) ~:~ |\Odd \cap A_* \cap S| + |\Even \cap A_* \cap S^c| \le m \big\} .\]
Note that, by Lemma~\ref{lem:tightness} and~\eqref{eq:def-D}, if $A$ is a $t$-approximation for some $1 \le t < 2d$ then
\begin{equation}\label{eq:tightness}
\OC{n}(A) \subset\OC{\lfloor n/(2d-t) \rfloor}^*(A) .
\end{equation}

\begin{lemma}\label{lem:approx-algorithm}
For any integers $m \ge 1$ and $1 \le t < 2d$ and any
approximation $A$, there exists a family $\cA$ of $t$-approximations
such that $\OC{m}^*(A) \subset \cut(\cA)$ and $|\cA| \le \exp(C \log d \cdot m/t)$.
\end{lemma}
\begin{proof}
We may assume $\OC{m}^*(A) \neq \emptyset$ as otherwise the
statement is trivial. Moreover, by Lemma~\ref{lem:eliminate isolated} we may
assume that $A_*$ has no isolated vertices.
For an indepedent set $W \subset A_*$ (i.e., a set containing no two adjacent vertices), write $W_{\Even}:=W\cap \Even$ and $W_{\Odd}:=W\cap \Odd$, and define
\begin{align*}
W_\ins &:= W_{\Even}^+ \cup\big(\Odd\cap N_t(A_*\setminus W^+)\big) ,\\
W_\out &:=  W_{\Odd}^+ \cup \big(\Even\cap N_t(A_*\setminus W^+)\big) .
\end{align*}
Here one should think of $W$ as recording the locations of a subset of even vertices in $A_*\cap S$ and odd vertices in $A_*\cap S^c$. We shall see that if this subset is chosen suitably then $W_\ins \subset S$ and $W_\out \subset S^c$.

Let $\cW$ denote the family of such sets $W$ having size at most $m/t$, and define
\[ \cB := \left\{ (A_\ins \cup W_\ins, A_\out \cup W_\out) ~:~ W \in \cW \right\} .\]
It is straightforward to check that every $B\in\cB$ is an approximation. Let us show that, for any $B\in \cB$, the maximal degree of the subgraph induced by $B_*$ is less than $t$, i.e., that $B_* \cap N_t(B_*) = \emptyset$.
Indeed, letting $W \in \cW$ be such that $B=(A_\ins \cup W_\ins,A_\out \cup W_\out)$ and noting that $B_* = A_* \setminus (W_\ins \cup W_\out)$ and $W_\ins \cup W_\out = W^+ \cup N_t(A_* \setminus W^+)$, we have
\[ B_* \cap N_t(B_*) \subset
   (A_* \setminus N_t(A_* \setminus W^+)\big) \cap N_t(A_* \setminus W^+) = \emptyset ,\]
Hence, applying Lemma~\ref{lem:eliminate isolated} to every element in $\cB$, we obtain a family $\cA$
of $t$-approximations such that $|\cA|\le|\cB|$
and $\cut(\cA)=\cut(\cB)$.

Next, we bound the size of $\cA$.
By assumption, there exists a set $S\in \OC{m}^*(A)$. Since $S$ is odd and since $A_*$ has no isolated vertices, we have
\[ A_* \subset (\Odd\cap A_* \cap S)^+ \cup (\Even \cap A_* \cap S^c)^+ ,\]
so that $|A_*| \le (2d+1)m \le 3dm$, and hence,
\[ |\cA| \le |\cB| \le |\cW| \le \sum_{k=0}^{\lfloor m/t \rfloor}
\binom{3dm}{k} \le (m/t + 1) (3edt)^{m/t} \le d^{Cm/t} .\]

It remains to show that $\cut_m^*(A) \subset \cut(\cB)$.
Let $S \in \cut_m^*(A)$ and let $W$ be a maximal subset of
$A_*$ among those satisfying $W_{\Even}\subset S$, $W_{\Odd}\subset S^c$ and $|A_* \cap N(W)| \ge t|W|$.
Observe that
\[ A_* \cap N(W) \subset (\Odd \cap A_* \cap S) \cup  (\Even \cap A_* \cap S^c) .\]
Thus, $t|W| \le |A_* \cap N(W)| \le m$ and, as $W$ is clearly an independent set, we have $W \in \cW$.
Now define $B_\ins := A_\ins \cup W_\ins$
and $B_\out := A_\out \cup W_\out$ so that $B:=(B_\ins,B_\out) \in \cB$.
We are left with showing that $S \in \cut(B)$, i.e., that $B_\ins \subset S$ and $B_\out \subset S^c$.
The two statements are very similar so we only show the former.
Since $S \in \cut(A)$, it suffices to show that $A_* \cap W_\ins \subset S$.
Let $v\in A_* \cap W_\ins$. If $v \in W_{\Even}^+$
then $v\in S$ by the definition of $W$ and since $S$ is odd. Otherwise, $v\in\Odd \cap (A_* \setminus W^+) \cap N_t(A_*\setminus W^+)$
so that $|A_* \cap N(W \cup\{v\})|\ge |A_* \cap N(W)|+t$.
Hence, $v \in S$ by the maximality of $W$.
\end{proof}

We are now ready to prove Lemma~\ref{lem:family-of-full-approx}.

\begin{proof}[Proof of Lemma~\ref{lem:family-of-full-approx}]
	Applying Lemma~\ref{lem:approx-algorithm} with $m_1:=m$ and $t_1=d$ to $A$, we obtain a family $\cB$ of $d$-approximation such that $|\cB| \le \exp(C \log d \cdot m/d)$ and $\OC{m}^*(A) \subset \cut(\cB)$.
	Applying Lemma~\ref{lem:approx-algorithm} with $m_2:=\lfloor n/d \rfloor$ and $t_2:=t$ to each $B \in \cB$, we obtain a collection of families of $t$-approximations.
	Taking the union over this collection, we obtain a family $\cA$ of
$t$-approximations such that $\cut^*_{\lfloor n/d \rfloor}(\cB) \subset \cut(\cA)$ and
	\[ |\cA| \le |\cB| \cdot \exp(C \log d \cdot n/dt) \le \exp(C \log d \cdot (m/d + n/dt)) .\]
	It remains to show that $\OC{n}(A) \subset \cut(\cA)$.
	To this end, let $S \in \cut_n(A)$ and note that $S \in \OC{m}^*(A) \subset \cut(\cB)$. Thus, there exists $B \in \cB$ such that $S \in \cut(B)$. Hence,
$S \in \OC{n}(B)$ and~\eqref{eq:tightness} now implies that
$S \in \OC{\lfloor n/d \rfloor}^*(B) \subset \cut(\cA)$, as required.
\end{proof}

\bibliographystyle{amsplain}
\bibliography{biblio}

\end{document}